\documentclass[12pt]{amsart}
\usepackage{amsmath,amsthm,amssymb,amscd,url,enumerate,mathtools}
\usepackage{graphicx}
\usepackage[margin=1in]{geometry}
\usepackage{afterpage}
\usepackage{tikz}
\usepackage{todonotes}
\usepackage[all]{xy}
\usepackage[colorlinks=true,citecolor={blue},linkcolor = {purple},
pdfauthor={Hanson Smith}, pdftitle={Prime Splitting in Radical Extensions}, pdfsubject={Algebraic Number Theory},
            pdfkeywords={Radical extension, Pure extension, Prime splitting, Prime ideal factorization, Monogenic}]{hyperref} 
\usepackage{tabularx}
\usepackage{subcaption} 
\usepackage{float}		
\usepackage{booktabs}
\usepackage{subcaption}
\usepackage{comment}
\usepackage{caption}

\usepackage{multicol}
\usepackage{relsize} 

\usepackage{dutchcal}
\usepackage{enumitem} 

\usepackage[square]{natbib}
\setcitestyle{numbers}

\newcommand{\TITLE}{Prime Splitting and Common Index Divisors in Radical Extensions}
\newcommand{\TITLERUNNING}{Prime Splitting in Radical Extensions}

\theoremstyle{plain}
\newtheorem{theorem}{Theorem}

\newtheorem{proposition}[theorem]{Proposition}
\newtheorem{lemma}[theorem]{Lemma}
\newtheorem{corollary}[theorem]{Corollary}

\theoremstyle{definition}
\newtheorem{definition}[theorem]{Definition}

\theoremstyle{remark}
\newtheorem{remark}[theorem]{Remark}

\newtheorem{example}[theorem]{Example}

\numberwithin{theorem}{section}

  {\end{list}}

  {\end{list}}

\newcommand{\tightoverset}[2]{%
  \mathop{#2}\limits^{\vbox to -.5ex{\kern-1.05ex\hbox{$#1$}\vss}}}

\numberwithin{equation}{section} 

\newcommand{\gm}{{\mathfrak{m}}}
\newcommand{\gn}{{\mathfrak{n}}}
\newcommand{\gp}{{\mathfrak{p}}}

\newcommand{\gI}{{\mathfrak{I}}}

\newcommand{\gP}{{\mathfrak{P}}}

\def\Ocal{{\mathcal O}}
\def\Pcal{{\mathcal P}}

\def\pcal{\mathcal{p}}

\newcommand{\FF}{\mathbb{F}}

\newcommand{\QQ}{\mathbb{Q}}

\newcommand{\ZZ}{\mathbb{Z}}

\newcommand{\tensor}{\otimes}
\newcommand{\dnd}{\nmid}

\newcommand{\ol}[1]{\overline{#1}}

\newcommand{\Disc}{\operatorname{Disc}}

\newcommand{\Irr}{\operatorname{Irred}}
\newcommand{\red}{\operatorname{red}}

\newcommand{\LPO}[4]{#1\Ocal_{#2\left(\hspace{-.5ex}\sqrt[#3]{#4}\right)}}

\title[\TITLERUNNING]{\TITLE}

\author[Hanson Smith]{Hanson Smith}
\address{Department of Mathematics, California State University San Marcos,
333 S. Twin Oaks Valley Rd.
San Marcos, CA 92096
USA}
\email{hsmith@csusm.edu}


\begin{document}

\sloppy 


\baselineskip=17pt


\begin{abstract}
We explicitly describe the splitting of odd integral primes in the radical extension $\mathbb{Q}(\sqrt[n]{a})$, where $x^n-a$ is an irreducible polynomial in $\mathbb{Z}[x]$. Our motivation is to classify common index divisors, the primes whose splitting provides a local obstruction to the existence of a power integral basis for the ring of integers of $\mathbb{Q}(\sqrt[n]{a})$. Among other results, we show that if $p$ is such a prime, even or otherwise, then $p$ divides $n$. 
\end{abstract}


\maketitle




\section{Introduction and main results}

The following is a classic theorem of Dedekind that is based on work of Kummer.
\begin{theorem}[Dedekind--Kummer Factorization]\label{Thm: DedekindKummerFactorization} Let $f(x)\in \ZZ[x]$ be monic and irreducible, and let $K=\QQ(\alpha)$, where $\alpha$ is a root of $f(x)$. If $p\in \ZZ$ is a prime that does not divide $\big[\Ocal_K:\ZZ[\alpha]\big]$, then the factorization of $p\Ocal_K$ mirrors the factorization of $f(x)$ in $\FF_p[x]$. More specifically, if
\[\ol{f(x)} = \ol{\phi_1(x)}^{e_1}\cdots\ol{\phi_r(x)}^{e_r}\]
is a factorization into irreducibles in $\FF_p[x]$ with the overbar indicating reduction modulo $p$, then the prime ideal factorization of $p\Ocal_K$ is
\[p\Ocal_K=\gp_1^{e_1}\cdots \gp_r^{e_r},\] 
where $\gp_i=(\phi_i(\alpha),p)$ and the residue class degree of $\gp_i$ is equal to the degree of $\phi_i(x)$.
\end{theorem}
Given a polynomial generating a number field $K$, Theorem \ref{Thm: DedekindKummerFactorization} gives a convenient way to compute the factorization of all but finitely many integral primes in the ring of integers~$\Ocal_K$. As seen above, we say an ideal $\mathfrak{I}$ \textit{mirrors} the factorization of a polynomial $f(x)$ if, given a factorization into irreducibles $f(x)=\prod_{i=1}^r \phi_i(x)^{e_i}$, there is a prime ideal factorization $\mathfrak{I}=\prod_{i=1}^r\gp_i^{e_i}$ where the residue class degree of $\gp_i$ is equal to the degree of $\phi_i(x)$.

This paper is focused on extensions generated by an irreducible polynomial of the shape $f(x)=x^n-a$. We call these \textit{radical extensions}\footnote{Radical extensions are also called \textit{pure extensions} or \textit{root extensions} in the literature.}, and we let $\sqrt[n]{a}$ denote an arbitrary root. The discriminant of the radical polynomial $f(x)$ is $\Disc(f)=\pm n^na^{n-1}$. Dedekind--Kummer factorization and the formula 
\[\Disc(f)=\Disc\Big(\QQ\big(\hspace{-.5ex} \sqrt[n]{a}\big)\Big)\cdot \Big[\Ocal_{\QQ(\hspace{-.5ex} \sqrt[n]{a})}:\ZZ\left[\sqrt[n]{a}\right]\Big]^2\] show that one can find the factorization of a prime $p\dnd na$ in the ring of integers $\Ocal_{\QQ(\hspace{-.5ex} \sqrt[n]{a})}$ by simply factoring $x^n-a$ in $\FF_p[x]$. The goal of this paper is to provide an explicit description of the factorization of the odd primes dividing $na$ and to use that description to classify the local obstructions to the monogenicity of $\QQ\big(\hspace{-.5ex} \sqrt[n]{a}\big)$. The explicit description, our main theorem, is the following:


\begin{theorem}\label{Thm: Main}
Let $p$ be an odd\footnote{The splitting of $p=2$ will be described in forthcoming work of the author and CSUSM master's student Dylan Scofield.} prime dividing $na$. The following three cases describe the factorization of the ideal generated by $p$ in the ring of integers of $\QQ\big(\hspace{-.5ex} \sqrt[n]{a}\big)$: 
\begin{enumerate}[label=\Roman*., ref=\Roman*]
\item If $p\mid a$ and either $p\nmid n$ or $p\mid n$ but $p\nmid v_p(a)$, we factor $y^{\gcd(v_p(a),n)}-a/p^{v_p(a)}$ into irreducibles in $\FF_p[y]$:
\[y^{\gcd(v_p(a),n)}-\frac{a}{p^{v_p(a)}}=\gamma_1(y)\cdots\gamma_r(y).\]
Then we have the prime ideal factorization
\[p\Ocal_{\QQ\left(\hspace{-.5ex}\sqrt[ n]{a}\right)} = \gp_1^{n / \gcd(v_p(a),n)}\cdots \gp_r^{n / \gcd(v_p(a),n)},\]
where each $\gp_i$ has residue class degree equal to the degree of $\gamma_i(y)$. \label{MainI}

\item If $p\mid n$ and $p\nmid a$, we define $w=v_p\big(a^{p}-a\big)$ and $n=n_0p^{m}$ where $m= v_p(n)$. We factor $x^{n_0}-a$ into irreducibles in $\FF_p[x]$:
\[x^{n_0}-a= \phi_1(x)\cdots \phi_r(x).\]
If we write $b = \min(w-1,m)$, then we have the prime ideal factorization
\[\LPO{p}{\QQ}{n}{a} = \prod_{i=1}^{r}\left( \gp_{i}^{p^{m- b}}\prod_{j=m- b+1}^{m}\gp_{i,j}^{\varphi(p^{j})} \right),\]
where $\varphi$ is Euler's phi function and each $\gp_i$ and $\gp_{i,j}$ has residue class degree equal to the degree of $\phi_i(x)$. If $w=1$, then the product is empty and taken to be 1. \label{MainII} 

\item If $p$ divides $n$, $a$, and $v_p(a)$, write $a=a_0p^{hp^k}$ where $\gcd(a_0,p)= 1$. Let $w_0=v_p\big(a_0^p-a_0\big)$, $c=\min(w_0-1,k,m)$, $g_0=\gcd(n_0,h)$, and $g=\gcd\big(n_0,h(p-1)\big)$. Then we have the factorization
\[\LPO{p}{\QQ}{n}{a} = \gI_0^{\frac{p^{m-c} n_0}{g_0}}\prod_{i=1}^c\gI_i^{\frac{p^{m-c}\varphi(p^i) n_0 }{g}},\]
where the factorization of $\gI_0$ mirrors the factorization of $R_{S_0}(y) = y^{g_0}-a_0$ in $\FF_p[y]$ and that of each $\gI_i$ with $i>0$ mirrors $R_S(y) = y^{g}-(-1)^{hp^k}a_0$ in $\FF_p[y]$. \label{MainIII}
\end{enumerate}
\end{theorem}


The proof of Theorem \ref{Thm: Main} proceeds via a variety of techniques, most prominently a $p$-adic factorization algorithm initialized by Ore and completed by Montes. 
Our methods give a precise description of the shape of the prime ideal factorization, but they are not sufficient to provide generators for the prime ideals as in Theorem \ref{Thm: DedekindKummerFactorization}. In some cases this is possible. For example, in case \ref{MainI} if $\gcd(v_p(a),n)=1$, then $p$ is totally ramified, and one can show $p\Ocal_{\QQ(\hspace{-.5ex} \sqrt[n]{a})}=\big(p,\sqrt[n]{a}\big)^n$. However, we have been unable to reverse through the $p$-adic approximations of the radical polynomial's factors provided by the Montes algorithm in a sufficiently generic way to construct generators in all cases.


The explicit description of splitting in Theorem \ref{Thm: Main} gives us a tool to classify the ``local" obstructions to monogenicity.

\begin{definition}\label{Def: commonindexdivisor}
Let $K/\QQ$ be a number field and write $\Ocal_K$ for the ring of integers. An integral prime $p$
is a \textit{common index divisor}\footnote{Common index divisors are also called \textit{essential discriminant divisors} and \textit{inessential or nonessential discriminant divisors}. The shortcomings of the English nomenclature are partly due to what Neukirch \cite[page 207]{Neukirch} calls ``the untranslatable German catch phrase [...] \textit{au{\ss}erwesentliche Diskriminantenteile}." 
See the final pages of Keith Conrad's exposition \href{https://kconrad.math.uconn.edu/blurbs/gradnumthy/dedekind-index-thm.pdf}{\textit{Dedekind's Index Theorem}} for a detailed explanation of the seemingly contradictory nomenclature.} for the extension $K/\QQ$ if 
\[p \text{ divides } \big[\Ocal_K:\ZZ[\alpha]\big]  \text{ for each }  \alpha\in\Ocal_K  \text{ with }  \QQ(\alpha)=K .\] 
\end{definition}

\begin{remark}\label{Rmk: LocalGlobalGlobal}
A common index divisor $p$ is a local obstruction to monogenicity in the sense that $\Ocal_K\tensor_\ZZ \ZZ_p\not\cong \ZZ[x]/f(x)\tensor_\ZZ \ZZ_p$ for any $f(x)\in\ZZ[x]$. See Example \ref{Ex: 5*7^2} for a cubic extension that is locally monogenic at each prime $p$ but not globally monogenic. Note that \cite{AlpogeBhargavaShnidman} consider both common index divisors and extensions as in Example \ref{Ex: 5*7^2} local obstructions since in both cases the equation formed by setting the index form equal to $\pm 1$ admits no solutions. 
\end{remark}

Hensel \cite{Hensel1894} connected common index divisors with prime splitting:
\begin{theorem}\label{Thm: HenselCIDs} The integral prime $p$ is a common index divisor of the number field $K$ if and only if there is a positive integer $f$ such that the number of prime ideal factors of $p\Ocal_K$ with residue class degree $f$ is greater than the number of monic irreducibles of degree $f$ in $\FF_p[x]$.
\end{theorem}

Gauss's formula for the number of monic irreducible polynomials of degree $f$ over $\FF_p$ is
\[\Irr(f,p)\coloneqq \frac{1}{f}\sum_{d\mid f}\mu\left(\frac{f}{d}\right)p^d, \text{ where } \mu \text{ is the M\"obius function.}\]

In the short note \cite{Zylinski}, von {\.Z}yli{\'n}ski uses Gauss's formula to prove that if $p$ is a common index divisor of a number field of degree $n$, then $p<n$.

As a consequence of our main theorem, we are able to classify odd common index divisors (CIDs) of $\QQ\big(\hspace{-.5ex} \sqrt[n]{a}\big)$. First, we state a simpler corollary of Theorem \ref{Thm: pmidanotn} that holds for all potential CIDs including $p=2$.

\begin{corollary}\label{Cor: quickCID}
Let $p$ be an integral prime that is not necessarily odd. If $p$ is a common index divisor of $\QQ\big(\hspace{-.5ex} \sqrt[n]{a}\big)$, then $p\mid n$. 
\end{corollary}

Combining Corollary \ref{Cor: quickCID} with von {\.Z}yli{\'n}ski's fact that a common index divisor is less than the degree, we have the following:

\begin{corollary}\label{Cor: NoPrimeDegCID}
If $n$ is prime, then $\QQ\big(\hspace{-.5ex} \sqrt[n]{a}\big)$ has no common index divisors.
\end{corollary}

The following is a full classification of odd common index divisors of $\QQ\big(\hspace{-.5ex} \sqrt[n]{a}\big)$ in terms of factorizations and counts of irreducible polynomials in $\FF_p[x]$.

\begin{corollary}\label{Cor: CIDs}
Let $p$ be an odd integral prime. Keep the notation of Theorem \ref{Thm: Main}.
\begin{enumerate}[label=\roman*., ref=\roman*]
\item If $p\mid a$ and either $p\dnd n$ or $p\mid n$ and $p\nmid v_p(a)$, then $p$ is not a common index divisor of $\QQ\big(\hspace{-.5ex} \sqrt[n]{a}\big)$. \label{CorCIDsi}

\item If $p\mid n$ and $p\nmid a$, then let $d_f$ be the number of irreducible factors of degree $f$ in the factorization of $x^{n_0}-a$ into irreducibles in $\FF_p[x]$.
The prime $p$ is a common index divisor of $\QQ\big(\hspace{-.5ex} \sqrt[n]{a}\big)$ if and only if 
\[\min(w,m+1) \cdot d_f>\Irr(f,p) \text{ for some positive integer }f.\] \label{CorCIDsii}

\item If $p$ divides $n$, $a$, and $v_p(a)$, then let $d_{f,0}$ be the number of irreducible factors of degree $f$ in the factorization of $y^{\gcd(n_0,h)}-a_0\in \FF_p[y]$ and let $d_f$ be the number of irreducible factors of degree $f$ in the factorization of $y^{\gcd(n_0,h(p-1))}-(-1)^{h}a_0 \in \FF_p[y].$
The prime $p$ is a common index divisor of $\QQ\big(\hspace{-.5ex} \sqrt[n]{a}\big)$ if and only if 
\[d_{f,0}+\min(w_0-1,k,m) \cdot d_f > \Irr(f,p) \text{ for some positive integer } f.\]\label{CorCIDsiii}
\end{enumerate}
\end{corollary}

\begin{proof}
First, we note that a common index divisor must divide $na$ since the factorization of other primes mirrors the factorization of a polynomial in $\FF_p[x]$.\footnote{This is seen more readily from analyzing the discriminant of $x^n-a$; however, we wanted to emphasize the polynomial factorization perspective.} In other words, if the splitting of $p$ coincides with the splitting of a polynomial in $\FF_p[x]$, then Theorem \ref{Thm: HenselCIDs} shows $p$ cannot be a common index divisor because its splitting is described by irreducibles in $\FF_p[x]$ of the requisite degree.  

Likewise, if $p\mid na$ but satisfies either of the conditions in \ref{CorCIDsi}, then Theorem~\ref{Thm: Main} case \ref{MainI} shows that the splitting of $p$ coincides with the splitting of a polynomial in $\FF_p[x]$. There are additional ramification indices; however, since Theorem \ref{Thm: HenselCIDs} is only concerned with distinct factors and residue class degrees, $p$ is not a common index divisor. Cases \ref{CorCIDsii} and \ref{CorCIDsiii} come from applying Theorem \ref{Thm: HenselCIDs} and Gauss's formula to the splittings given in Theorem \ref{Thm: Main} cases \ref{MainII} and \ref{MainIII}. 
\end{proof}


As the following example demonstrates, Corollary \ref{Cor: CIDs} allows us to construct number fields with a specified common index divisor. 

\begin{example}\label{Ex. SpecificCID}
If $p$ is an odd prime, we can construct a radical extension with $p$ as a common index divisor as follows. Let $n=p^p$ and $a=1+p^{p+1}$. Since $w=p+1$ and $m=p$, we have $b=\min(w-1,m)=p$. Further, $n_0=1$, so $x^{n_0}-a$ is linear, and $r=1$. Thus, Theorem~\ref{Thm: Main} case \ref{MainII} shows $p\Ocal_{\QQ(\hspace{-.5ex} \sqrt[n]{a})}$ splits into $p+1$ distinct primes, each having residue class degree~1. As summarized by Corollary \ref{Cor: CIDs} case \ref{CorCIDsii}, there are only $p$ distinct linear polynomials in $\FF_p[x]$, so the splitting of $p$ cannot be obtained from the splitting of an irreducible polynomial in $\FF_p[x]$.  Thus $p$ is a common index divisor of $\QQ\big(\hspace{-.5ex} \sqrt[n]{a}\big)/\QQ$. Indeed, \[\min(w,m+1)\cdot d_1=(p+1)\cdot 1>\Irr(1,p)=p.\] 
\end{example}

See Examples \ref{Ex: x^27-80} and \ref{Ex: 729} for a similar but more explicit constructions. Theorem \ref{Thm: HenselCIDs} shows that it is splitting (in particular, distinct prime ideal factors of same degree) as opposed to ramification that forces an integral prime to become a common index divisor. Theorem \ref{Thm: Main} and Corollary \ref{Cor: CIDs} together show that potential common index divisors in a radical extension experience a good deal of ramification. For this reason, a relatively large degree radical extension is needed for a given common index divisor to occur. The following example demonstrates, for a given integral prime $p$, a construction of an extension of minimal degree where $p$ is a common index divisor. Using \cite{Pleasants}, we can generalize this example to construct rings of integers of minimal degree requiring an arbitrary number of ring generators. This will be detailed in forthcoming work.

\begin{example}\label{Ex. LowDegCIDs}
    Fix a prime $p$. Letting $\ell$ be any prime not equal to $p$, the extension given by a root of
    \[x^{p+1}+\ell p x^{p}+\ell p^3x^{p-1}+\cdots + \ell p^{\frac{i(i+1)}{2}}x^{p-i+1}+\cdots \ell p^{p(p+1)/2}x+\ell p^{(p+1)(p+2)/2}\] has $p$ as a common index divisor. Indeed the polynomial is $\ell$-Eisenstein hence irreducible. Applying the Montes algorithm, the principal $x$-polygon has $p+1$ sides, each with length 1 and slopes descending from $p+1$ to 1. See Figure \ref{Fig. p+1sides}.


\begin{figure}[h!]
\begin{tikzpicture}
      \draw[<->] (0,5.5) -- (0,0) -- (4.5,0);
      \draw [fill] (0,5) circle [radius = .05];

      \node [left] at (-.2,5) {$\big(0,v_3(2\cdot 3^{10})\big)$};
      
      \draw [fill] (1,3) circle [radius = .05];
      \node [right] at (.95,3.2) {$\big(1,v_3(2\cdot 3^6)\big)$};
	  \draw [fill] (2,1.5) circle [radius = .05]; 
	  \node [right] at (1.95,1.9) {$\big(2,v_3(2\cdot 3^3)\big)$};     
      \draw [fill] (3,.5) circle [radius = .05];
      \node [right] at (2.9,.9) {$\big(3,v_3(2\cdot 3)\big)$};   
      \draw [fill] (4,0) circle [radius = .05];
      \node [above] at (4.4,-.75) {$\big(0,v_3(1)\big)$};
      \foreach \x in  {1,2,3,4} {
      		\draw (\x, 2pt) -- +(0,-4pt);
      }
      \foreach \y in  {.5,1,1.5,2,2.5,3,3.5,4,4.5,5} {
      		\draw (2pt, \y) -- +(-4pt, 0);
      }
      
      \draw (0,5) -- (1,3);
      \draw (1,3) -- (2,1.5);
      \draw (2,1.5) -- (3,.5);
      \draw (3,.5) -- (4,0);
	\end{tikzpicture}
  \caption{Principal $x$-polygon for $x^4-2\cdot 3 x^3-2\cdot 3^3 x^2-2\cdot 3^6 x - 2\cdot 3^{10}$}
  \label{Fig. p+1sides}
\end{figure}

\end{example}


One can ask about the power of a common index divisor dividing the index of each monogenic order. Ore conjectured \cite{Ore} and Engstrom proved \cite{Engstrom} that the power of a common index divisor is not determined by the prime ideal decomposition. See \cite{Nart} and \cite{DelCorsoDvornicich} for more recent developments.

\begin{example}\label{Ex: 5*7^2}
Common index divisors are local obstructions to monogenicity; however, there are also global obstructions. Consider the well-known example of $f(x) = x^3-5\cdot 7^2$. Corollary \ref{Cor: CIDs} and the fact that $2\dnd \Disc f(x)$ shows that there are no common index divisors of the extension $\QQ\big(\hspace{-.5ex} \sqrt[3]{5\cdot 7^2}\big)/\QQ$. Indeed, since $\big\{1,\sqrt[3]{5\cdot 7^2}, \sqrt[3]{5^2\cdot  7}\big\}$ is an integral basis, we see 
\[\bigg[\Ocal_{\QQ\left(\hspace{-.5ex} \sqrt[3]{5\cdot 7^2}\right)}:\ZZ\left[\sqrt[3]{5\cdot 7^2}\right]\bigg]=7 \ \text{ and } \ \bigg[\Ocal_{\QQ\left(\hspace{-.5ex} \sqrt[3]{5\cdot 7^2}\right)}:\ZZ\left[\sqrt[3]{5^2\cdot 7}\right]\bigg]=5. \]

The index form, which is $\pm 1$ if and only if $X\sqrt[3]{5\cdot 7^2}+ Y\sqrt[3]{5^2\cdot  7}$ is a monogenerator, is $5Y^3-7X^3$. The equation $5Y^3-7X^3=\pm 1$ has no integer solutions, so $\QQ\big(\hspace{-.5ex} \sqrt[3]{5\cdot 7^2}\big)$ is not monogenic despite being locally monogenic (for each $p$, there is a power $\ZZ_p$-integral basis) at each prime $p$. The interested reader should consult \cite{GaalsBook} and \cite{EvertseGyoryBook} for more on index form equations and \cite{ABHS} for an algebro-geometric perspective. Note, the prime 7 is a local obstruction in the sense of \cite{AlpogeBhargavaShnidman} as $5Y^3-7X^3=\pm 1$ admits no solutions over $\QQ_7$. 

Zack Wolske has brought it to the author's attention that the example $\QQ\big(\hspace{-.5ex} \sqrt[3]{23\cdot 15^2}\big)$ from \cite{SpearmanYangYoo} has a global obstruction in the sense of \cite{AlpogeBhargavaShnidman}. Using Magma, \cite{SpearmanYangYoo} shows the index form $23X^3-15Y^3=\pm 1$ has no integer solutions. However, the rational solution $(\frac12,\frac12)$ and Hensel's lemma applied to $23X^3=1$ for $p=2$ show there are $\ZZ_p$ solutions for each $p$. 
\end{example}


\subsection{The proof of Theorem \ref{Thm: Main}}

The following list provides a proof of Theorem \ref{Thm: Main} from a bird's-eye view; the difficulty is in establishing the results referenced. When not too cumbersome, we prove more general results than are needed to obtain Theorem \ref{Thm: Main} since they may be of interest to future authors.  

\begin{enumerate}
\item[1.] Case \ref{MainI} is established by specializing Theorem \ref{Thm: pmidanotn} to $\QQ$ and taking $p$ as our uniformizer.

\item[2.] Case \ref{MainII} is Corollary \ref{Cor: psplittinginnQ}, a specialization of Proposition \ref{Prop: pdividingnandnotageneral} to $\QQ$.

\item[3.] Case \ref{MainIII} is Proposition \ref{Prop: SplittingnotExplicit} with $K=\QQ$ in addition to some explicit work on uniformizers as well as a simplification of the relevant residual polynomials.
\end{enumerate}


\subsection{Outline of the paper}

In Section \ref{Sec: PrevWork} we outline some previous results in a few relevant areas. Section \ref{Sec: Montes} details the Montes algorithm, which will be our primary tool for proving Theorem \ref{Thm: Main}. Section \ref{Sec: Prelim} surveys a number of results regarding radical extensions that we will employ. Section \ref{Section: PrimePowerExtensions} establishes the factorization of an odd prime $p$ in the extension obtained from $x^{p^m}-a$ when $p\nmid a$. This is already proven by V\'elez in the case where the base field is $\QQ$, see Theorem \ref{Thm: VelezMain}. However, we prove it for a more general class of number fields, and our methods highlight the role played by the valuation of the Wieferich difference, see Definition~\ref{Def: WieferichS4}. Section \ref{Sec: Primesdividingabutnotgcd} establishes Theorem \ref{Thm: pmidanotn}. This theorem describes the factorization of primes $\gp$ dividing $a$ but not $\gcd\big(a,v_\gp(n)\big)$ in the extension obtained by adjoining $\sqrt[n]{a}$. Notably, this result holds with any number field as the base field and includes $\gp$ of even residue characteristic. Section \ref{Section: Primesdividingnnota} describes primes dividing $n$ but not $a$. The key ingredient is Theorem \ref{Thm: psplittingp^mgeneral}, so we are able to adopt the same generality as that used in Section \ref{Section: PrimePowerExtensions}. Section \ref{Sec: DifficultCase} confronts the remaining case of primes $\gp$ dividing $a$ and $\gcd\big(v_\gp(a),n\big)$. This is the most difficult case from the perspective of the Montes algorithm, i.e., Newton polygon techniques. We employ techniques similar to V\'elez to factor $\gp$ in the extension obtain by adjoining the $p^{v_\gp(n)}$-th root of $a$. With this in hand, the Montes algorithm can complete the factorization in $K\big(\hspace{-.5ex} \sqrt[n]{a}\big)$. Though we work in some generality, it is necessary for us to add additional hypotheses for this case. In particular, we require that the localization of $K$ at $\gp$ does not contain any $p$-th roots of unity. Examples are interspersed throughout: \ref{Ex. SpecificCID}, \ref{Ex. LowDegCIDs}, \ref{Ex: 5*7^2}, \ref{Ex: x^27-80}, \ref{Ex: 729}, \ref{Ex: x^5*27-80}, \ref{Ex: 2727}, \ref{Ex: 5^5*26}, \ref{Ex: p=3Knumberfield}, \ref{Ex: ppower1}, \ref{Ex: a=3 to the 2*27 times 80}, \ref{Ex: a=3 to the 5*27 times 26}. There are four figures. Figure \ref{Fig. p+1sides} helps visualize Example \ref{Ex. LowDegCIDs}, while Figures \ref{Fig: pdoesntdividea} and \ref{Fig: xpolygeneraldegree} help visualize the relevant Newton polygons in Theorems \ref{Thm: psplittingp^mgeneral} and \ref{Thm: pmidanotn}, respectively. Figure \ref{Fig: SectionStrategy} surveys our proof strategy in Section \ref{Sec: DifficultCase}.


\subsection*{Acknowledgments}

The author would like to thank Keith Conrad and Jordi Gu\`ardia for the helpful conversations. The author would also like to thank the anonymous referee for the useful comments and suggestions.


\section{Previous work}\label{Sec: PrevWork}

There has been a good deal of recent work on common index divisors. See \cite{MR2251710}, \cite{MR3889325}, \cite{MR4736305}, \cite{MR4563435}, \cite{MR2531162}, \cite{MR4510645}, \cite{MR731633}, as well as other works from these authors for various families of polynomials/fields of a fixed degree. When consulting the literature, it bears noting that the question of whether or not a given a binomial (radical) polynomial or a trinomial polynomial is monogenic has been fully answered. See \cite{g17} and \cite{JKS1}, respectively. For families of varying degree, \cite{PethoPohst} study multiquadratic fields, while \cite{SpearmanWilliamsYang} studying fields generated by polynomials of prime degree with dihedral Galois groups.

We note that \cite{Berwick1927} uses Newton polygon techniques to establish an integral basis for $\QQ(\hspace{-.5ex} \sqrt[n]{a})$. In some cases, this work can describe the splitting of primes dividing $na$, but it does not fully describe splitting.  
The general method is subsumed by earlier work of Ore \cite{OreReview}. 

In \cite{ObusWild}, the author computes bounds on the conductors of extensions obtained by a root of unity and a radical. In a particular case when $p=2$, exact values of the conductor are computed. \cite{Viviani} constructs a uniformizer for $\QQ_p\big(\zeta_p, \sqrt[p^{i+1}]{a}\big)/\QQ_p\big(\zeta_p, \sqrt[p^{i}]{a}\big)$. \cite{BellemareLei} describes a uniformizer for the extension $\QQ_p\big(\zeta_{p^2},\sqrt[p]{p}\big)/\QQ_p$, while \cite{WangYuan} continues down this avenue by constructing a uniformizer for $\QQ_p\big(\zeta_{p^2},\sqrt[p^m]{p}\big)/\QQ_p$ with $m\geq 1$. 

In \cite{Velez1/p}, V\'elez describes the factorization of a prime above $p$ in any extension of a number field obtained by adjoining a $p^{\text{th}}$ power or a $p^{\text{th}}$ root of unity. In the series of papers \cite{MannVelez} and \cite{VelezCoprime}, prime splitting in an arbitrary radical extensions is described in a variety of cases. However, the cases are not exhaustive, especially when $p$ divides the degree. Moreover, the phrasing (in terms of products of prime ideal factor counting functions) is not as conducive to our study of common index divisors. 

In \cite{VelezPrimePower}, V\'elez completely describes the splitting of the prime $p$ in a $p$-power radical extension of $\QQ$. We summarize this for odd primes in Theorem \ref{Thm: VelezMain}. There is some overlap between this work and our present study; however, we employ different methods and have a different goal. V\'elez's primary application was describing the genus field of $\QQ\big(\hspace{-.5ex} \sqrt[n]{a}\big)$, while our application is the classification of common index divisors of $\QQ\big(\hspace{-.5ex} \sqrt[n]{a}\big)$. In addition, our use of the Montes algorithm allows us to tackle radical extensions of arbitrary degree and phrase our results explicitly in terms of valuations. In summary, though there has been a variety of previous studies of prime splitting in radical extensions, the novelty of our current undertaking is that it is complete and self-contained. 


\section{The Montes algorithm and a theorem of Ore}\label{Sec: Montes}


The Montes algorithm is a $p$-adic factorization algorithm that is based on and extends the pioneering work of {\O}ystein Ore \cite{Ore}. We will essentially only employ the aspects developed by Ore here, but we will use the notation and setup of the general implementation. For the complete development of the Montes algorithm, see \cite{GMN}. Our notation will roughly follow \cite{ElFadilMontesNart}, which gives a more extensive summary than we undertake here. One can also consult \cite{JhorarKhanduja}.

Let $p$ be an integral prime, $K$ a number field with ring of integers $\Ocal_K$, and $\gp$ a prime of $K$ above $p$. Write $K_\gp$ to denote the completion of $K$ at $\gp$. By a \textit{uniformizer at $\gp$} or a \textit{uniformizer of $K_\gp$}, we mean an element $\pi_\gp\in\Ocal_K$ such that $v_\gp\big(\pi_\gp\big)=1$. 
Suppose we have a monic, irreducible polynomial $f(x)\in \Ocal_K[x]$. We extend the standard $\gp$-adic valuation to $\Ocal_K[x]$ by defining the $\gp$-adic valuation of $f(x) = a_n x^n + \cdots + a_1 x + a_0 \in \Ocal_K[x]$ to be 
	\[ v_\gp\big(f(x)\big) = \min_{0 \leq i \leq n} \big( v_\gp(a_i) \big). \]
This is sometimes called the \textit{Gauss valuation.}
If $\phi(x), f(x) \in \Ocal_K[x]$ are monic and such that $\deg \phi \leq \deg f$, then we can write
    	\[f(x)=\sum_{i=0}^k a_i(x)\phi(x)^i,\]
for some $k$, where each $a_i(x) \in \Ocal_K[x]$ has degree less than $\deg \phi$. We call the above expression the \emph{$\phi$-adic development} of $f(x)$. We associate to the $\phi$-adic development of $f(x)$ an open Newton polygon by taking the lower convex hull
of the integer lattice points $\big(i,v_p(a_i(x))\big)$. The sides of the Newton polygon with negative slope are the \emph{principal $\phi$-polygon}. 

Write $k_\gp$ for the residue field $\Ocal_K/\gp$, and let $\ol{f(x)}$ be the image of $f(x)$ in $k_\gp[x]$. It will often be the case that we develop $f(x)$ with respect to an irreducible factor $\phi(x)$ of $\ol{f(x)}$. In this situation, we will want to consider the extension of $k_\gp$ obtained by adjoining a root of $\phi(x)$. We denote this finite field by $k_{\gp,\phi}$. We associate to each side of the principal $\phi$-polygon a polynomial in $k_{\gp,\phi}[y]$. Suppose $S$ is a side of the principal $\phi$-polygon with initial vertex $\big(s,v_\gp(a_s(x))\big)$, terminal vertex $\big(k,v_\gp(a_k(x))\big)$, and slope $-\frac{h}{e}$ written in lowest terms. Define the length of the side to be $l(S)=k-s$ and the degree to be $d\coloneqq\frac{l(S)}{e}$. Let $\red:\Ocal_K[x]\to k_{\gp,\phi}$ denote the homomorphism obtained by quotienting by the ideal $\big(\gp,\phi(x)\big)$.
For each $i$ in the range $b\leq i\leq k$, we define the residual coefficient to be
\[c_i=\left\{
\begin{array}{ll}
0 \text{ if }  \big(i,v_\gp(a_i(x))\big)  \text{ lies strictly above } S  \text{ or } v_\gp(a_i(x))=\infty,\\
\red\left(\frac{a_i(x)}{\pi^{v_\gp(a_i(x))}}\right)  \text{ if }  \big(i,v_\gp(a_i(x))\big) \text{ lies on } S.
\end{array}
\right.\]
Finally, the \emph{residual polynomial} of the side $S$ is the polynomial
\[R_S(y)=c_s+c_{s+e}y+\cdots +c_{s+(d-1)e}y^{d-1}+c_{s+de}y^d\in k_{\gp,\phi}[y].\]
Notice, that $c_s$ and $c_{s+de}$ are always nonzero since they are the initial and terminal vertices, respectively, of the side $S$. In this work, we will almost always be developing $f(x)$ with respect to a linear polynomial, so $k_{\gp,\phi}=k_\gp$, and we will often write the latter to ease notation.


Having established notation, we state a theorem that connects prime splitting and polynomial factorization. The ``three dissections" that we will outline below are due to Ore, and the full Montes algorithm is an extension of this. Our statement loosely follows Theorem 1.7 of \cite{ElFadilMontesNart}. 

\begin{theorem}\label{Thm: Ore}[Ore's Three Dissections]
Let $f(x)\in \Ocal_K[x]$ be a monic irreducible polynomial and let $\alpha$ be a root. Suppose
\[\ol{f(x)}=\phi_1(x)^{r_1}\cdots \phi_s(x)^{r_s}.\]
is a factorization into irreducibles in $k_\gp[x]$. Hensel's lemma shows $\phi_i(x)^{r_i}$ corresponds to a factor of $f(x)$ in $K_\gp[x]$ and hence to a factor $\gm_i$ of $\gp \Ocal_{K(\alpha)}$. 

Choose a lift of $\phi_i(x)$ to $\Ocal_K[x]$ and, abusing notation, call this lift $\phi_i(x)$. Developing $f(x)$ with respect to $\phi_i(x)$, suppose the principal $\phi_i$-polygon has sides $S_1,\dots, S_g$. Each side of this polygon corresponds to a distinct factor of $\gm_i$. 

Write $\gn_j$ for the factor of $\gm_i$ corresponding to the side $S_j$. Suppose $S_j$ has slope $-\frac{h}{e}$. If the residual polynomial $R_{S_j}(y)$ is separable, then the prime factorization of $\gn_j$ mirrors the factorization of $R_{S_j}(y)$ in $k_{\gp,\phi_i}[y]$, but every factor of $R_{S_j}(y)$ will have an exponent of $e$. In other words,
\[\text{if } R_{S_j}(y)=\gamma_1(y)\dots\gamma_k(y) \ \text{ in }  \ k_{\gp,\phi_i}[y], \ \text{ then } \ \gn_j = \gP_1^{e}\cdots \gP_k^{e}  \ \text{ in } \ \Ocal_{K(\alpha)},\]
with $\deg(\gamma_m)$ equaling the residue class degree of $\gP_m$ for each $1\leq m\leq k$.
In the case where $R_{S_j}(y)$ is not separable, further developments are required to factor $\gp$. 
\end{theorem}


\section{Preliminaries}\label{Sec: Prelim}

In this section we review and establish a few results that aid our description of prime splitting in radical extensions. We often focus on $\QQ$ to make our discussion more concise and because our goal is to describe splitting in $\QQ\big(\hspace{-.5ex} \sqrt[n]{a}\big)$; however, in later sections we will attempt to be as general as our methods permit.

If $v_p(n)=m$, then the splitting of $p$ in $\QQ\big(\hspace{-.5ex} \sqrt[p^m]{a}\big)$ is a hurdle that must be overcome to obtain the splitting of $p$ in $\QQ\big(\hspace{-.5ex} \sqrt[n]{a}\big)$. In order to surmount this, the following factorization in key:

\begin{equation}\label{Eq: binomexpansionS4}
\begin{split}
x^{p^m}-a&=\left(x-a+a\right)^{p^m}-a\\
&=\left(\sum\limits_{k=0}^{p^m}\binom{p^m}{k}\left(x-a\right)^k a^{p^m-k}\right) -a\\
&=\left(\sum\limits_{k=1}^{p^m}\binom{p^m}{k} a^{p^m-k} \left(x-a\right)^k\right) + a^{p^m}-a.
\end{split}
\end{equation}

The analysis of the expansion in \eqref{Eq: binomexpansionS4} motivates the following lemmas.

\begin{lemma}\label{Lemma: ValBinom}
The $p$-adic valuation of $\binom{p^m}{b}=\binom{p^m}{p^k-b}$ is $m-v_p(b)$.
\end{lemma}

\begin{proof}
We have
\[\binom{p^m}{b}=\frac{p^m (p^m-1) \cdots (p^m-(b-1))}{b (b-1)\cdots 1}.\]
Note that $v_p(p^m-c)=v_p(c)$ for all $1\leq c\leq p^m$. Hence, the $p$-adic valuation of $\binom{p^m}{b}$ is $v_p\left(p^m\right)-v_p(b).$
\end{proof}

For convenience and as an homage to Arthur Wieferich, we make the following definition. This definition will be generalized in the next section.

\begin{definition}\label{Def: WieferichS4}
Define the \textit{Wieferich difference} (of $a$ with respect to $p^m$) to be $a^{p^{m}}-a$. The $p$-adic valuation of this difference is key to describing the splitting of $p$. We denote this valuation with $w$: 
\[w\coloneqq v_p\left(a^{p^{m}}-a\right).\] 
\end{definition}

The valuation of the Wieferich difference does not depend on $m$. 

\begin{lemma}\label{Lem: ExpDoesntMatter}
Let $a\in \ZZ$, then 
\[v_p\left(a^p-a\right) = v_p\left(a^{p^m}-a\right)\]
for every $m>0$.
\end{lemma}

\begin{proof}
If $p\mid a$, then this is clear. Suppose $p\nmid a$. It suffices to show that 
\[v_p\left(a^{p-1}-1\right) = v_p\left(a^{p^m-1}-1\right)\]
The smallest of Fermat's theorems tells us that the base-$p$ expansion of $a^{p-1}$ has the form 
\[a^{p-1}=1+a_wp^w+(\text{higher powers of }p)\]
where each $a_i$ is in the range $0< a_i <p$. Clearly, 
\[v_p\left(a^{p-1}-1\right)=v_p\left(a_wp^w+(\text{ higher powers of }p)\right)=w.\]
Note $p^m-1=(p-1)(p^{m-1}+p^{m-2}+\cdots+p+1)$
,so
\begin{align*}
a^{p^m-1} &= \left( a^{p-1}\right)^{p^{m-1}+p^{m-2}+\cdots+p+1}\\
&=\big(1+a_wp^w+(\text{higher powers of }p)\big)^{p^{m-1}+p^{m-2}+\cdots+p+1}\\
&=1 +\left(p^{m-1}+p^{m-2}+\cdots+p+1\right) a_wp^w + (\text{higher powers of }p).
\end{align*}
We can now see that 
\[v_p\left(a^{p^m-1}-1\right)=v_p\left(a_wp^w+(\text{higher powers of }p)\right)=w. \qedhere \]
\end{proof}

Notice that this proof will hold, mutatis mutandis, for an arbitrary prime $\gp$ of residue characteristic $p$ in some number field $K$ so long as we require that either $\gp\mid a\Ocal_K$ or $a^p\equiv a\bmod \gp$. 

\begin{remark}\label{Rmk: AlmostCyclotomicPolys}
In many ways the behavior of radical extensions agrees with the behavior that we are accustomed to in cyclotomic extensions. It is this analogy that motivates the clean proofs in \cite{VelezPrimePower}.

Following V\'elez, define $s$ to be such that $a\in \QQ_p^{p^s}$ but $a\notin \QQ_p^{p^{s+1}}$. In other words, $a$ is a $p^s$-power in $\QQ_p$ but not a $p^{s+1}$-power. Later, we will explicitly describe $s$ in terms of a valuation. When $p$ does not divide both $a$ and $v_p(a)$, then $s$ is simply one less than the valuation of the Wieferich difference: $v_p(a^p-a)-1$.

Write $\Phi_{p^j}\big(x,\sqrt[p^s]{a}\big)$ for the ``twisted cyclotomic polynomial" whose roots are $\zeta_{p^j}^k \sqrt[p^{s}]{a}$ with $\zeta_{p^j}^k$ primitive. Explicitly, if $\zeta_{p^j}$ is a primitive $p^j$-th root of unity, then
\[\Phi_{p^j}\left(x,\sqrt[p^s]{a}\right)=\prod_{\substack{{1\leq k< p^j} \\ {\gcd(k,p)=1}}} x-\zeta_{p^j}^k \sqrt[p^{s}]{a}.\]
When $\sqrt[p^{s}]{a}\in \ZZ_p$, then we have the factorization
\[x^{p^{s}}-a=x^{p^{s}}-\left(\hspace{-.5ex} \sqrt[p^{s}]{a}\right)^{p^{s}}=\left(x-\sqrt[p^{s}]{a}\right)\prod_{1\leq k\leq s} \Phi_{p^k}\left(x,\sqrt[p^s]{a}\right). \]

When $p^m>p^{s}$, the factorization of $x^{p^m}-a$ in $\ZZ_p[x]$ is 
\[x^{p^m}-a= \left(x^{p^{m-s}}\right)^{p^{s}}-\left(\hspace{-.5ex} \sqrt[p^{s}]{a}\right)^{p^{s}} =\left(x^{p^{m-s}}-\sqrt[p^{s}]{a}\right)\prod_{s\leq k \leq m} \Phi_{p^k}\left(x^{p^{m-w}},\sqrt[p^s]{a}\right) .\]

It is this clever factorization and a lemma about ramification in the compositum of a cyclotomic field and a radical extension that V\'elez uses to give a clean proof of the factorization of the odd prime $p$ in the extension $\QQ\big(\hspace{-.5ex} \sqrt[p^m]{a}\big)$. 
\end{remark}

Summarizing Theorems 2 and 5 of \cite{VelezPrimePower}:

\begin{theorem}\label{Thm: VelezMain}
If $s\geq m$, then 
\[  p\Ocal_{\QQ\left(\hspace{-.5ex}\sqrt[ p^m]{a}\right)}  = \gp_0\big(\gp_1\gp_2^p\cdots \gp_m^{p^{m-1}}\big)^{p-1}.\] 
If $s< m$, then
\[\LPO{p}{\QQ}{p^m}{a} =\gp_0^{p^{m-s}}\big(\gp_1\gp_2^p\cdots \gp_s^{p^{s-1}}\big)^{(p-1)p^{m-s}}.\] 
\end{theorem}

Though our proofs with Newton polygons are more involved, they allow for more generality as well as a description that depends only on $p$-adic valuations.  


We will also use a classical result on the irreducibility of radical polynomials called the Vahlen--Capelli Theorem. For a reference see \cite{Turnwald} or Chapter 6, \S 9 of \cite{Algebra}.

\begin{theorem}\label{Thm: VahlenCapelli}
Let $K$ be a field and $x^n-a\in K[x]$. Then $x^n-a$ is reducible over $K$ if and only if for some prime $p\mid n$ we have $a\in K^p$ or $4\mid n$ and $a\in -4K^4$.
\end{theorem}


\section{The factorization of primes above $p$ and not dividing $a$ in $K\big(\hspace{-.5ex} \sqrt[p^m]{a}\big)$}\label{Section: PrimePowerExtensions}

Though our main goal is a description of the splitting of odd primes in an arbitrary radical extension of $\QQ$, we will work in a more general situation here since this setup will be required in Sections \ref{Section: Primesdividingnnota} and \ref{Sec: DifficultCase} and because the more general results are interesting in their own right. Let $K$ be a number field and let $\gp$ be a prime of $K$ above $p$. Write $e_\gp$ to denote the ramification index of $\gp$ over $p$. Suppose $\gcd(e_\gp,p)=1$; i.e., $\gp$ is not wildly ramified over $p$. Write $f$ for the residue class degree; i.e., $|\Ocal_K/\gp|=p^f$. We consider an irreducible polynomial $x^{p^m}-a$ in $\Ocal_K[x]$ and we suppose $\gp$ is prime to $a\Ocal_K$. In this section we aim to explicitly describe the factorization of $\gp$ in the ring of integers of $K\big(\hspace{-.5ex} \sqrt[p^m]{a}\big)$.

When $f=1$, then $a$ is a $p^m$-th root of $a$ modulo $\gp$; however, we will construct a $p^m$-th root of a generic $a$ here. Let $\mu\equiv m\bmod f$ be such that $1\leq\mu \leq f$. Now $a^{p^{f-\mu}}$ is a $p^m$-th root of $a$ in $\Ocal_K/\gp$. Say $m=kf+\mu$, so
\[\left(a^{p^{f-\mu}}\right)^{p^m} = a^{p^{f-\mu+m}} = a^{p^{f(k+1)}} \equiv a \bmod \gp.\]

Ultimately, we want the factorization of $x^{p^m}-a$ in $K_\gp[x]$, where $K_\gp$ is the completion of $K$ at $\gp$. Proceeding with the Montes algorithm, we start by reducing modulo $\gp$:
\[x^{p^m}-a\equiv \left(x - a^{p^{f-\mu}}\right)^{p^m} \bmod \gp.\]
Thus, we need to take the $(x - a^{p^{f-\mu}})$-adic development.

\begin{equation}\label{Eq: binomexpansion}
\begin{split}
x^{p^m}-a&=\left(x-a^{p^{f-\mu}}+a^{p^{f-\mu}}\right)^{p^m}-a\\
&=\left(\sum\limits_{k=0}^{p^m}\binom{p^m}{k}\left(x-a^{p^{f-\mu}}\right)^k\left(a^{p^{f-\mu}}\right)^{p^m-k}\right) -a\\
&=\left(\sum\limits_{k=1}^{p^m}\binom{p^m}{k}\left(a^{p^{f-\mu}}\right)^{p^m-k}\left(x-a^{p^{f-\mu}}\right)^k\right) + a^{p^{f-\mu+m}}-a.
\end{split}
\end{equation}

We see that the behavior of the principal $\big(x-a^{p^{f-\mu}}\big)$-polygon depends on the valuation of $a^{p^{f-\mu+m}}-a$, so we generalize Definition \ref{Def: WieferichS4}.

\begin{definition}\label{Def: Wieferich}
Define the \textit{Wieferich difference} (of $a$ with respect to $p^m$ and $\gp$) to be $a^{p^{f-\mu+m}}-a$. We will be particularly interested in the $p$-adic valuation of this difference, which we will denote by 
\[w\coloneqq v_p\left(a^{p^{f-\mu+m}}-a\right).\] 
We have suppressed $a$, $p^m$, and $\gp$ in the notation since context will make these clear. 
\end{definition}

Definition \ref{Def: Wieferich} formalizes the extent to which a lift of a root of $x^{p^m}-a$ modulo $\gp$ remains a root of $x^{p^m}-a$. With our definition of the Wieferich difference solidified, the following theorem demonstrates that the factorization of $\gp\Ocal_{K(\hspace{-.5ex} \sqrt[p^m]{a})}$ depends heavily on $w$. 


\begin{theorem}\label{Thm: psplittingp^mgeneral} Let $\gp$ be a prime ideal in $\Ocal_K$ above the odd prime $p$, and let $x^{p^m}-a\in \Ocal_K[x]$ be irreducible with $v_{\gp}(a)=0$ and $w$ as above. Suppose the ramification index $e_\gp$ is not divisible by $p$. 
Let $l$ denote $\lceil\frac{w}{e_\gp}-\frac{p}{p-1}\rceil$. If $l \leq 0$, then suppose $p\dnd w$, and if $l < m$, then suppose $p \dnd \big(e_\gp l-w\big)$. 
Write $b=\min(l,m)$, then we have the prime ideal factorization 
\[\gp\Ocal_{K\left(\hspace{-.5ex}\sqrt[ p^m]{a}\right)} = \gP^{p^{m - b} }  \prod_{i=m-b+1}^{m} \gI_i^{\varphi(p^{i})/\gcd(e_\gp,p-1)},   \]
where $\varphi$ is Euler's phi function. If $b\leq 0$, then the empty product is taken to be 1 and $\gp$ is totally ramified in $\Ocal_{K(\hspace{-.5ex} \sqrt[p^m]{a})}$. Further, writing $\pi_\gp$ for a uniformizer at $\gp$, the factorization of the ideal\footnote{Here we label ideals with the exponent of $p$ in the $x$-coordinate of the terminal (right-most) vertex of the side of the principal $\big(x - a^{p^{f-\mu}}\big)$-polygon they correspond to. Proposition \ref{Prop: DifficultCase} employs a different labeling.}  
$\gI_i\subset \Ocal_{K(\hspace{-.5ex} \sqrt[p^m]{a})}$ mirrors the factorization of 
\[ \frac{\binom{p^m}{p^{i-1}}}{\pi_\gp^{e_\gp(m-i+1)}}\left(a^{p^{f-\mu}}\right)^{p^m-p^{i-1}} + \frac{\binom{p^m}{p^{i}}}{\pi_\gp^{e_\gp(m-i)}}\left(a^{p^{f-\mu}}\right)^{p^m-p^{i}} \ y^{\gcd(e_\gp,p-1)} \in k_{\gp}[y]. \]
\end{theorem}

Excluding number fields $K$ with wild ramification above $p$ is necessary for our methods. When the numerator of the slopes of the relevant principal polygon is divisible by $p$, one must continue through the Montes algorithm in a manner that is often difficult to do generically. We will employ different methods to deal with this phenomenon in Section \ref{Sec: DifficultCase}.

\begin{proof}
We will use the Montes algorithm to factor $x^{p^m}-a$ over $K_\gp[x]$. We have 
\[x^{p^m}-a\equiv \left(x-a^{p^{f-\mu}}\right)^{p^m} \bmod \gp.\]
Equation \eqref{Eq: binomexpansion} yields the $\big(x - a^{p^{f-\mu}}\big)$-adic development. The lower convex hull of the points corresponding to the valuations of the coefficients of this development is the principal $\big(x-~a^{p^{f-\mu}}\big)$-polygon.


\begin{figure}[h!]
\begin{tikzpicture}

      \draw[<->] (0,4.5) -- (0,0) -- (9.5,0);
      \draw [fill] (0,4) circle [radius = .05];

      \node [left] at (-.2,4) {$w=4$};
      
      \draw [fill] (1/3,3) circle [radius = .05];
      \node [right] at (.5,3.2) {$\left(1,v_3\binom{27}{1}\right)$};
	  \draw [fill] (1,2) circle [radius = .05]; 
	  \node [right] at (1.2,2.3) {$\left(3,v_3\binom{27}{3}\right)$};     
      \draw [fill] (3,1) circle [radius = .05];
      \node [right] at (3.15,1.35) {$\left(9,v_3\binom{27}{9}\right)$};   
      \draw [fill] (9,0) circle [radius = .05];
      \node [above] at (9,-.7) {$27$};
      \foreach \x in  {1,2,3,4,5,6,7,8,9} {
      		\draw (\x, 2pt) -- +(0,-4pt);
      }
      \foreach \y in  {1,2,3,4} {
      		\draw (2pt, \y) -- +(-4pt, 0);
      }
      
      \draw (0,4) -- (1/3,3);
      \draw (1/3,3) -- (1,2);
      \draw (1,2) -- (3,1);
      \draw (3,1) -- (9,0);
	\end{tikzpicture}
  \caption{Example principal $(x - a)$-polygon corresponding to $(3)=\gP\gP_1^2\gP_2^6\gP_3^{18}$}
  \label{Fig: pdoesntdividea}
\end{figure}


The first vertex of the principal $\big(x-a^{p^{f-\mu}}\big)$-polygon of $x^{p^m}-a$ is $(0,w)$ and the last vertex is $(p^m,0)$. To begin, we will investigate when the polygon is one-sided. From Lemma \ref{Lemma: ValBinom} and since $\gp\dnd a\Ocal_K$, the possible candidates for the terminal vertex of the leftmost side of the principal $\big(x-a^{p^{f-\mu}}\big)$-polygon vertices are
\[\left\{\left(1,v_\gp\binom{p^m}{1}\right), \left(p,v_\gp\binom{p^m}{p}\right), \dots , \left(p^{m-1},v_\gp\binom{p^m}{p^{m-1}}\right),\left(p^{m},v_\gp\binom{p^m}{p^{m}}\right)\right\},\footnote{We have conflated binomial and valuation parentheses for readability.}\]
which we can rewrite as simply
\[\left\{\left(1,e_\gp m\right), \big(p,e_\gp(m-1)\big), \dots , \left(p^{m-1},e_\gp\right), \left(p^{m},0\right)\right\}.\] 

The possible slopes of the first side are 
\[
\left\{  e_\gp m - w, \frac{e_\gp (m-1)-w}{p}, \frac{e_\gp(m-2)-w}{p^2}, \dots , \frac{e_\gp-w}{p^{m-1}}, -\frac{w}{p^m} \right\},
\]
with the last possibility corresponding to a one-sided polygon. Thus the principal $\big(x-a^{p^{f-\mu}}\big)$-polygon is one-sided if and only if $-\frac{w}{p^m}\leq \frac{e_\gp(m-i)-w}{p^i}$ for all $0\leq i<m$. 
This is equivalent to $\frac{w}{e_\gp} \leq \frac{m-i)}{p^{m-i}-1}+m-i$ for all $0\leq i<m$. The minimum value is achieved when $i=m-1$. Thus, the principal $\big(x-a^{p^{f-\mu}}\big)$-polygon is one-sided if and only if $\frac{w}{e_\gp}\leq \frac{p}{p-1}\iff l = \lceil\frac{w}{e_\gp}-\frac{p}{p-1}\rceil\leq 0$. In this case, $\gp$ is totally ramified in $K\big(\hspace{-.5ex} \sqrt[p^m]{a}\big)$ since our hypothesis is that $p\dnd w$. 


Henceforth, we will assume $l>0$, so the principal $\big(x-a^{p^{f-\mu}}\big)$-polygon has at least two sides. The leftmost side $S_1$ originates at $(0,w)$. 
To simplify the exposition, we will first deal with the case where $S_1$ has length 1. This occurs exactly when $e_\gp m-w<\frac{e_\gp(m-1)-w}{p}$, which simplifies to $l=\lceil\frac{w}{e_\gp}-\frac{p}{p-1}\rceil\geq m$. 
We see $b = m$, and the terminal vertex of $S_1$ is $(1,e_\gp m)$. Hence, $S_1$ corresponds to an unramified, degree 1 prime of $K\big( \hspace{-.5ex}\sqrt[p^m]{a} \big)$ above $\gp$. 

Continuing in the case where $l\geq m$, the second side $S_2$ of the principal $\big(x-a^{p^{f-\mu}}\big)$-polygon terminates at $\big(p,e_\gp(m-1)\big)$ and has slope $-\frac{e_\gp}{p-1}$. Calculating from left to right, the residual polynomial associated to the second side is 
\[R_{S_2}(y)=\frac{\binom{p^m}{1}}{\pi_\gp^{e_\gp m}}\left(a^{p^{f-\mu}}\right)^{p^m-1} + \frac{\binom{p^m}{p}}{\pi_\gp^{e_\gp (m-1)}} \left(a^{p^{f-\mu}}\right)^{p^m-p} y^{\gcd(e_\gp,p-1)} \in k_{\gp,x-a^{p^{f-\mu}}}[y] = k_\gp[y]  .\] 
Thus we have uncovered the partial factorization $\gP \gI_{1}^{(p-1)/\gcd(e_\gp,p-1)}$ where the prime factorization of $\gI_{1}$ mirrors the factorization of $R_{S_2}(y)$ in $k_\gp[y]$. Since the terminal vertex of $S_2$ is $\big(p,e_\gp (m-1)\big)$, the rest of our polygon will agree with the case where $l<m$ 
and the side $S_1$ has length $p$.

Suppose now that we are in the case where $l<m$, so $b=l$. 
The slopes of the sides of the principal $\big(x-a^{p^{f-\mu}}\big)$-polygon must be negative. Thus, as possibilities for the terminal vertex of the first side $S_1$ we have the points $\big(p^i, e_\gp(m-i)\big)$ with corresponding slopes $\frac{e_\gp(m-i)-w}{p^i}$ for $m-\frac{w}{e_\gp} < i < m$ or $m- \big(\lceil \frac{w}{e_\gp} \rceil-1\big) \leq i \leq m-1$. 

If $m- \big(\lceil \frac{w}{e_\gp} \rceil-1\big) \leq i,j \leq m-1$ with $i<j$, then we see that 
\begin{equation}\label{Eq: smallslopes}
\frac{e_\gp(m-i)-w}{p^i}<\frac{e_\gp(m-j)-w}{p^j} \iff m-i+\frac{j-i}{p^{j-i}-1}<\frac{w}{e_\gp}.
\end{equation}
Since $p$ is odd, $\frac{j-i}{p^{j-i}-1}<1$. Hence, if $m- \big(\lceil \frac{w}{e_\gp} \rceil-1\big)<i<j$, then \eqref{Eq: smallslopes} shows the slope corresponding to terminal vertex $\big(p^i, e_\gp(m-i)\big)$ is less than that corresponding to $\big(p^j, e_\gp(m-j)\big)$. 
Thus, we need only compare $i = m- \lceil \frac{w}{e_\gp} \rceil+1$ and $j= m- \lceil \frac{w}{e_\gp} \rceil+2$.
Equation \eqref{Eq: smallslopes} shows the terminal vertex for $S_1$ is 
\[\left(p^{m-\lceil \frac{w}{e_\gp}\rceil +1}, e_\gp\left(\left\lceil \frac{w}{e_\gp}\right\rceil-1\right)\right) \iff \left\lceil \frac{w}{e_\gp}\right\rceil-1 +\frac{1}{p-1}<\frac{w}{e_\gp}\iff l=\left\lceil \frac{w}{e_\gp}\right\rceil-1.\]
Conversely, the terminal vertex is 
\[\left(p^{m-\lceil \frac{w}{e_\gp}\rceil +2}, e_\gp\left(\left\lceil \frac{w}{e_\gp}\right\rceil-2\right)\right) \iff l=\left\lceil \frac{w}{e_\gp}\right\rceil-2.\]
Notice that in both cases we can write the terminal vertex as $\big(p^{m-l},e_\gp l\big)=\big(p^{ m- b},e_\gp b\big)$.
Hence, the slope of $S_1$ is 
$(e_\gp b - w)\big/p^{ m-  b}$. 
By hypothesis, $p$ does not divide $e_\gp l - w$. Thus, $S_1$ corresponds to a degree 1 prime $\gP$ above $\gp$ with ramification index $p^{m-  b}$. 

For the second side $S_2$, the possibilities for the terminal vertex are the points $\big(p^i, e_\gp(m-i)\big)$, with $m-  b + 1 \leq i \leq m$. The corresponding slopes are $\big(e_\gp(m- i) - e_\gp b \big)\big/ \big(p^i-p^{m-  b}\big)$. Here, the least value of $i$ results in the least slope, so the terminal vertex of the second side is $\big(p^{m-  b + 1}, e_\gp(  b - 1)  \big)$ and the slope is 
$-e_\gp \big/ p^{m-  b}(p-1)$.

Recalling, $p\dnd e_\gp$, we see the residual polynomial in $k_{\gp,x-a^{p^{f-\mu}}}[y]=k_\gp[y]$ associated to the side $S_2$ is 
\[
R_{S_2}(y)=\frac{\binom{p^m}{p^{m-  b}}}{\pi_\gp^{e_\gp b}} \left(a^{p^{f-\mu}}\right)^{p^{m}- p^{m-  b}} \\  
+  \frac{\binom{p^m}{p^{m-  b + 1}}}{\pi_\gp^{e_\gp( b - 1) }} \left(a^{p^{f-\mu}}\right)^{p^{m} - p^{m-  b +1} }  \ \  y^{\gcd(e_\gp,p-1)} .
\] 
Thus $S_2$ corresponds to a factor $\gI_{m-  b + 1}^{p^{m-  b}(p-1)/\gcd(e_\gp,p-1)}$ of $\gp$ in $K\big(\hspace{-.5ex} \sqrt[p^m]{a}\big)$ where the prime ideal factorization of $\gI_{m-  b + 1}$ mirrors the factorization of $R_{S_2}(y)$ into irreducibles in $k_\gp[y]$.

One continues this process to achieve a principal $\big(x-a^{p^{f-\mu}}\big)$-polygon with $ b+1$ sides and slopes $(e_\gp b-w)\big/p^{ m-  b }$ and $-e_\gp \big/p^{m- b+j}(p-1)$ with $0\leq j\leq  b-1$. 

For example, if $ b=l\geq 2$, the third side $S_3$ will have slope $-e_\gp\big/p^{m-  b +1}(p-1)$ and residual polynomial
\[
R_{S_3}(y)=\frac{\binom{p^m}{p^{m-  b + 1}}}{\pi_\gp^{e_\gp( b - 1)}} \left(a^{p^{f-\mu}}\right)^{p^{m}- p^{m-  b + 1}} 
+  \frac{\binom{p^m}{p^{m-  b + 2}}}{\pi_\gp^{e_\gp( b - 2) }} \left(a^{p^{f-\mu}}\right)^{p^{m} - p^{m-  b + 2} }  \ \  y^{\gcd(e_\gp,p-1)} \text{ in } k_\gp[y].
\] 
Here $S_3$ corresponds to an ideal factor $\gI_{m-  b + 2}^{p^{m-  b +1} (p-1) / \gcd(e_\gp,p-1)}$ in $K\big(\hspace{-.5ex} \sqrt[p^m]{a}\big)$ where the splitting of $\gI_{m-  b +2}$ mirrors the splitting of $R_{S_3}(y)$ into irreducibles in $k_\gp(y)$.

All that remains before concluding with our desired factorization of $\gp$ is to consider the separability of the residual polynomials attached to each side. Our hypotheses ensure that all of the residual polynomials are radical polynomials of degree coprime to $p$. Thus, they are separable.
\end{proof}


For clarity and utility, we will state Theorem \ref{Thm: psplittingp^mgeneral} for the special case when $e_\gp=1$. 

\begin{corollary}\label{Cor: psplittingp^m1} Let $\gp$ be a prime of a number field $K$ above the odd prime $p$, and suppose $\gp$ is unramified over $p$. Take $x^{p^m}-a$ in $\Ocal_K[x]$ irreducible and having $v_{\gp}(a)=0$. Let $w$ be as in Definition \ref{Def: Wieferich} and write $ b=\min(w-1,m)$. Then we have the prime ideal factorization 
\[\LPO{\gp}{K}{p^m}{a} = \gP^{p^{m-b}}\prod_{i=m-b+1}^{m}\gP_i^{\varphi(p^i)}, \]
where the empty product when $w=1$ is taken to be $1$.
\end{corollary}


\begin{example}\label{Ex: x^27-80}
Consider $x^{27}-80$. For primes $\ell$ not dividing $3\cdot 80$, Dedekind--Kummer factorization tells us that we can obtain the factorization of $\ell$ in $\QQ\big(\hspace{-.5ex} \sqrt[27]{80}\big)$ by simply factoring $x^{27}-80$ modulo $\ell$. For example, $x^{27}-80$ is irreducible in $\FF_7[x]$, so $7$ remains prime in $\QQ\big(\hspace{-.5ex} \sqrt[27]{80}\big)$ and has residue class degree 27.

Corollary \ref{Cor: psplittingp^m1} allows us to factor $3$ in $\QQ\big(\hspace{-.5ex} \sqrt[27]{80}\big)$. We compute that $w=v_3\big(80^{27}-80\big)=4$. See Figure \ref{Fig: pdoesntdividea}. Thus, $3$ splits into four primes with residue class degree 1 in $\QQ\big(\hspace{-.5ex} \sqrt[27]{80}\big)$. More precisely,
\[\LPO{3}{\QQ}{27}{80} =\gp\gp_1^2\gp_2^{6}\gp_3^{18}.\]
Hence $3$ is a common index divisor as there are only three linear polynomials in $\FF_3[x]$. This is of course generalized and streamlined by Corollary \ref{Cor: CIDs} case \ref{CorCIDsii}. There we simply note that $4>3$. As the degree is relatively small, one can readily confirm the factorization of $3\Ocal_{\QQ(\hspace{-.5ex} \sqrt[27]{80})}$ with SageMath.
\end{example}
 
\begin{example}\label{Ex: 729}
For a bit more novelty, we can consider $\QQ\big(\hspace{-.5ex} \sqrt[729]{2186}\big)$. SageMath is much less agreeable when asked to factor $3$ in this number field. However, we can compute that $w=v_3\big(2186^{729}-2186\big)=7$, and Corollary \ref{Cor: psplittingp^m1} ($4\cdot 1>3$) tells us
\[\LPO{3}{\QQ}{729}{2186} =\gp\gp_1^2\gp_2^{6}\gp_3^{18}\gp_4^{54}\gp_5^{243-81}\gp_6^{729-243}.\] 
Using the logic above, or appealing to Corollary \ref{Cor: CIDs} case \ref{CorCIDsii} (Indeed, $7>3$), the integral prime 3 is a common index divisor of $\QQ\big(\hspace{-.5ex} \sqrt[729]{2186}\big)$. In fact, a result of Pleasants \cite{Pleasants} shows that $\Ocal_{\QQ(\hspace{-.5ex} \sqrt[729]{2186})}$ requires three ring generators over $\ZZ$ because of the splitting of 3. Applications to the minimal number of ring generators are explored in a forthcoming work by the author and a master's student. 
\end{example}


\section{The factorization of primes dividing $a$ but not $\gcd\big(n,v_\gp(a)\big)$}\label{Sec: Primesdividingabutnotgcd}

We are ready to turn our attention to general radical extensions. As noted, we are most interested in describing prime splitting in an arbitrary radical extension of $\QQ$. However, we will aim for our intermediate results to be as general as possible. This section describes the most straightforward case. Notice that we make no assumptions on the residue characteristic in this section. In particular, the theorem below includes primes of even residue characteristic. 

\begin{theorem}\label{Thm: pmidanotn}
Suppose $x^n-a$ is an irreducible polynomial in $\Ocal_K[x]$, where $K$ is an arbitrary number field. Let $\gp$ be a prime of $\Ocal_K$ such that $\gp \mid a\Ocal_K$ but $p\nmid \gcd\big(v_\gp(a),n\big)$, where $p$ is the characteristic of $k_\gp$ the residue field of $\gp$.  

Letting $\pi_\gp$ be a uniformizer at $\gp$, if we have a factorization into irreducibles in $k_\gp[y]$:
\[y^{\gcd(v_\gp(a),n)}-\frac{a}{\pi_\gp^{v_\gp(a)}}=\gamma_1(y)\cdots\gamma_r(y),\]  
then we have the prime ideal factorization
\[\LPO{\gp}{K}{n}{a} = \gP_1^{n /\gcd(v_\gp(a),n)}\cdots \gP_r^{n / \gcd(v_\gp(a),n)},\]
where the degree of $\Ocal_{K(\hspace{-.5ex} \sqrt[n]{a})}/\gP_i$ over $k_\gp$ is equal to the degree of $\gamma_i(y)$. 
\end{theorem}

\begin{proof}
Reducing $x^n-a$ modulo $\gp$, we have $x^n$, so we take the principal $x$-polygon. An example of the shape of this polygon is shown Figure \ref{Fig: xpolygeneraldegree}. The single side $S$ of this polygon has slope $-\frac{v_\gp(a)}{n}$. Write this in lowest terms as $-\frac{h}{e}$ and notice $e=\frac{n}{\gcd(v_\gp(a),n)}$. 
We find the residual polynomial associated to the single side of the polygon is \[R_S(y)=y^{\gcd(v_\gp(a),n)}-a/\pi_\gp^{v_\gp(a)}.\] Since $p\dnd \gcd\big(v_\gp(a),n\big)$, the roots of unity of order $\gcd\big(v_\gp(a),n\big)$ are distinct in $\ol{\FF_p}$. Thus, $R_S(y)$ is separable in $k_{\gp,x}[y]=k_\gp[y]$, and Theorem \ref{Thm: Ore} yields the stated factorization.
\end{proof}


\begin{figure}[h!]
\begin{tikzpicture}

      \draw[<->] (0,2.5) -- (0,0) -- (10,0);
      \draw [fill] (0,2) circle [radius = .05];

      \node [right] at (-1.3,2.05) {$v_p(a)$};
      
      
	  \draw [fill] (5,1) circle [radius = .05];
      
      \draw [fill] (10,0) circle [radius = .05];
      \node [above] at (10,-.6) {$n$};
      \node [above] at (5.9,1.1) {\text{Slope } $= \ -\frac{v_p(a)}{n}$}; 
      \foreach \x in  {1,2,3,4,5,6,7,8,9,10} {
      		\draw (\x, 2pt) -- +(0,-4pt);
      }
      \foreach \y in  {1,2} {
      		\draw (2pt, \y) -- +(-4pt, 0);
      }
      
      \draw (0,2) -- (10,0);
	\end{tikzpicture}
  \caption{An example $x$-polygon for $p\mid a$. To be explicit here, one could take $p=5$, $n=10$, and $a=75$. The residual polynomial is $x^2-3$, so in $\QQ\big(\hspace{-.5ex} \sqrt[10]{75}\big)$ one has $(5)=\gp^5$ where $\gp$ has residue class degree 2. }
  \label{Fig: xpolygeneraldegree}
\end{figure}


In particular, if we ignore the ramification indices, the splitting of a rational prime $p$ dividing $a$ but not $n$ mirrors the splitting of the separable polynomial $y^{\gcd(v_p(a),n)}-a/p^{v_p(a)}$ in $\FF_p[y]$. Notice also that Theorem \ref{Thm: pmidanotn} holds for all primes not just primes of odd residue characteristic. Hence, we have the corollary stated in the introduction:

\vspace{.5 cm}

\noindent\textbf{Corollary \ref{Cor: quickCID}.} \textit{If the integral prime $p$ is a common index divisor of $\QQ\big(\hspace{-.5ex} \sqrt[n]{a}\big)$, then $p\mid n$.}\\


\section{The factorization of primes dividing $n$ but not $a$}\label{Section: Primesdividingnnota}


After the previous case, the next most straightforward situation is that of primes dividing $n$ but not $a$. Theorem \ref{Thm: psplittingp^mgeneral} will be key for our work here, so our hypotheses will mirror those. To be explicit, suppose $K$ is a number field, $x^n-a\in \Ocal_K[x]$ is irreducible, and $\gp$ is a prime of $K$ above the odd prime $p$ such that $p\mid n$ and $v_{\gp}(a)=0$. Suppose the ramification index $e_\gp$ of $\gp$ over $p$ is not divisible by $p$. Write $n=p^m n_0$ with $\gcd(p,n_0)=1$. 
Recall that $a^{p^{f-\mu}}$ is the explicit $p^m$-th root of $a$ in $k_\gp$ constructed in Section \ref{Section: PrimePowerExtensions} and $w = v_\gp\big((a^{p^{f-\mu}})^{p^m}-a\big)$. We define $l=\lceil\frac{w}{e_\gp}-\frac{p}{p-1}\rceil$. If $l \leq 0$, then suppose $p\dnd w$, and if $l < m$, then suppose $p \dnd (e_\gp l-w)$. We will prove the following.

\begin{proposition}\label{Prop: pdividingnandnotageneral}
With the hypotheses as above, take a factorization of $x^{n_0}-a$ into irreducibles in $k_\gp[x]$:
\begin{equation}\label{Eq: pfactors}
x^{n_0}-a = \phi_1(x)\cdots \phi_r(x). 
\end{equation}


If we define $ b=\min(l,m)$, then we have the ideal factorization 
\[\LPO{\gp}{K}{n}{a} = \prod_{i=1}^r\left( \gP_i^{p^{m-  b} }  \prod_{j=m-  b +1}^{m} \gI_{i,j}^{\varphi(p^{j})/\gcd(e_\gp,p-1)} \right),  \]
where if $l\leq 0$ the empty product is taken to be 1. The degree of $\Ocal_{K(\hspace{-.5ex} \sqrt[n]{a})}/\gP_i$ over $k_\gp$ is equal to the degree of $\phi_i(x)$, and the prime ideal factorization of the ideal
$\gI_{i,j}\subset \Ocal_{K(\hspace{-.5ex} \sqrt[n]{a})}$ mirrors the factorization of 
\begin{equation}\label{Eq: ResPolySec7}
\frac{\binom{p^m}{p^{j-1}}}{\pi_\gp^{e_\gp(m-j+1)}}\left(a^{p^{f-\mu}}\right)^{p^m-p^{j-1}} + \frac{\binom{p^m}{p^{j}}}{\pi_\gp^{e_\gp(m-j)}}\left(a^{p^{f-\mu}}\right)^{p^m-p^{j}} \ y^{\gcd(e_\gp,p-1)} \in k_{\gp}[x]\big/\big(\phi_i(x)\big)[y], 
\end{equation}
with a degree $d$ irreducible factor of the residual polynomial in \eqref{Eq: ResPolySec7} corresponding to a prime ideal factor of $\gp \Ocal_{K(\hspace{-.5ex} \sqrt[n]{a})}$ having a residue field of degree $d\cdot \deg\phi_i(x)$ over $k_\gp$.
\end{proposition}


\begin{proof}
Our strategy will be to factor $\gp$ in $\Ocal_{K(\hspace{-.5ex} \sqrt[{n_0} \ ]{a})}$ first and then apply Theorem \ref{Thm: psplittingp^mgeneral} to $x^{p^m}-a$ over $K\big(\hspace{-.5ex} \sqrt[n_0]{a}\big)$.

Since $\gp\dnd n_0a \Ocal_{K}$, Dedekind--Kummer factorization (Theorem \ref{Thm: DedekindKummerFactorization}) shows that the prime ideal factorization of $\gp\Ocal_{K(\hspace{-.5ex} \sqrt[{n_0} \ ]{a})}$ mirrors the factorization of $x^{n_0}-a$ in \eqref{Eq: pfactors}.

Letting $\Pcal_i$ be the prime ideal factor of $\gp \Ocal_{K(\hspace{-.5ex} \sqrt[{n_0}\ ]{a})}$ corresponding to $\phi_i(x)$, we note the ramification index of $\Pcal_i$ over $p$ is $e_\gp$, and $a^{p^{f-\mu}}$ remains $p^m$-th root of $a$ in $k_{\gp}[x]/\big(\phi_i(x)\big)$, the residue field of $\Pcal_i$. Hence, $w=v_\gp\big((a^{p^{f-\mu}}\big)^{p^m}-a)=v_{\Pcal_i}\big((a^{p^{f-\mu}})^{p^m}-a\big)$. Thus, the hypotheses of Theorem \ref{Thm: psplittingp^mgeneral} still hold. Applying this theorem yields the desired result.
\end{proof}

As one might expect, if $\Pcal_i\neq \Pcal_j$ are two primes of $\Ocal_{K(\hspace{-.5ex} \sqrt[n_0 \ ]{a})}$ above $\gp$, then the residual polynomials describing the splitting of these primes in the ring of integers of $K\big(\hspace{-.5ex} \sqrt[n_0]{a},\sqrt[p^m]{a}\big)$ are the same. The difference between the residue fields $k_{\gp}[x]/\big(\phi_i(x)\big)$ and $k_{\gp}[x]/\big(\phi_j(x)\big)$ will account for any difference in the splitting of $\Pcal_i$ and $\Pcal_j$.


The following corollary gives a simpler statement when $K=\QQ$. Note that when $a\in \ZZ$, Lemma \ref{Lem: ExpDoesntMatter} shows $w=v_p\big(a^{p^m}-a\big)=v_p\big(a^{p}-a\big)$.

\begin{corollary}\label{Cor: psplittinginnQ} With the setup and notation as above but $K=\QQ$ and $\gp=p$, 
we factor 
\begin{equation*}
x^{n_0}-a = \phi_1(x)\cdots \phi_r(x) \ \text{ in } \ \FF_p[x], \ \text{ so } \ p\Ocal_{\QQ(\hspace{-.5ex} \sqrt[{n_0 \ }]{a})} = \gp_1\cdots\gp_r. 
\end{equation*} 
Writing $ b=\min(w-1,m)$, the prime ideal factorization of $p\Ocal_{\QQ(\hspace{-.5ex} \sqrt[n]{a})}$ is
\[\LPO{p}{\QQ}{n}{a} = \prod_{i=1}^{r}\left( \gP_{i}^{p^{m- b}}  \prod_{j=m- b+1}^{m}\gP_{i,j}^{\varphi(p^{j})}  \right),\]
where each $\gP_i$ or $\gP_{i,j}$ has residue class degree equal to the degree of $\phi_i(x)$.
\end{corollary}


\begin{example}\label{Ex: x^5*27-80}
Consider $x^{5\cdot 27}-80$ and the number field $\QQ(\hspace{-.5ex} \sqrt[135]{80})$. Factoring into irreducibles,  
\[x^5-80 = (x + 1)  (x^4 + 2x^3 + x^2 + 2x + 1) \ \text{ in } \ \FF_3[x].\]
Hence, building on our work in Example \ref{Ex: x^27-80}, we find
\[\LPO{3}{\QQ}{135}{80} =\gp\gp_1^2\gp_2^{6}\gp_3^{18}\pcal\pcal_1^2\pcal_2^{6}\pcal_3^{18},\] 
where each $\gp$ has residue class degree 1 and each $\pcal$ has residue class degree 4. SageMath confirms this. Since the $\gp$'s have residue class degree 1, the prime $3$ is a common index divisor. The four distinct degree 1 prime ideal factors cannot correspond to four distinct degree 1 polynomials in $\FF_3[x]$.
\end{example}

\begin{example}\label{Ex: 2727}
To more clearly see the computational benefits of Corollary \ref{Cor: psplittinginnQ}, we factor  
\[x^{101}-80 = (x + 1)  (x^{100}+ 2x^{99} + \cdots + 2x + 1) \ \text{ in } \ \FF_3[x].\]
We find that 
\[\LPO{3}{\QQ}{2727}{80} =\gp\gp_1^2\gp_2^{6}\gp_3^{18} \pcal \pcal_1^2\pcal_2^{6}\pcal_3^{18},\] 
where each $\gp$ has residue class degree 1 and each $\pcal$ has residue class degree 100.
\end{example}


\section{The factorization of primes $p$ dividing $a$ when $p$ divides $\gcd\big(v_p(a),n\big)$}\label{Sec: DifficultCase}

This section confronts the most difficult case from the perspective of Newton polygon methods: $p\mid a$ and $p\mid \gcd\big(v_p(a),n\big)$. Writing $n=p^m n_0$ with $m=v_p(n)$, the first subsection establishes the factorization of $p$ in $\QQ\big(\hspace{-.5ex} \sqrt[p^m]{a}\big)$. The second subsection describes the factorization of the primes of $\QQ\big(\hspace{-.5ex} \sqrt[p^m]{a}\big)$ above $p$ in $\QQ\big(\hspace{-.5ex} \sqrt[p^m]{a},\sqrt[n_0]{a}\big)=\QQ\big(\hspace{-.5ex} \sqrt[n]{a}\big)$. 

In order to make this description explicit, it is necessary to construct uniformizers. Recall, we write $a=a_0p^{hp^k}$ with $p$ not dividing $a_0$ or $h$. We also have $w_0=v_p\big(a_0^{p^m}-a_0\big)=v_p\big(a_0^p-a_0\big)$. We will show (Proposition \ref{Prop: DifficultCase}) that all the splitting at $p$ in $\QQ\big(\hspace{-.5ex} \sqrt[p^m]{a}\big)$ happens in $\QQ\big(\hspace{-.5ex} \sqrt[p^c]{a}\big)$ where $c=\min(w_0-1,k,m)$. We will see that it is sufficient to work in $\QQ\big(\hspace{-.5ex} \sqrt[p^c]{a}\big)$. In this field, the local extensions are simply $p$-power cyclotomic extensions, making the necessary uniformizers particularly simple. 
 
The following diagram gives a road map:


\begin{figure}[h!]
\xymatrix{
&&&\QQ\big(\hspace{-.5ex}\sqrt[n_0]{a},\sqrt[p^m]{a}\big)=\QQ\big(\hspace{-.5ex}\sqrt[n]{a}\big)\ar@{-}[d]^{\text{Prop. \ref{Prop: SplittingnotExplicit} + details for $K=\QQ$.}} & & \gI_0^{\frac{p^{m-c} n_0}{g_0}}\prod\limits_{i=1}^c\gI_i^{\frac{p^{m-c}\varphi(p^i) n_0 }{g_i}}\ar@{-}[d]\\
&&&\QQ\big(\hspace{-.5ex}\sqrt[p^m]{a}\big)\ar@{-}[d]^{\text{Prop. \ref{Prop: DifficultCase}.}} & & \gp_0^{p^{m-c}} \prod\limits_{i=1}^{c}\gp_i^{p^{m-c}\varphi(p^i)}\ar@{-}[d]\\
&&&\QQ && (p)\\
}
\caption{Describing the splitting when $p\mid a$ and $p\mid \gcd(v_p(a),n)$}
\label{Fig: SectionStrategy}
\end{figure}


\subsection{Irreduciblility}\label{Subsec: Irred}

Since it is not much more difficult, we employ the generality of Proposition \ref{Prop: pdividingnandnotageneral}: Let $K$ be a number field, $x^{p^m}-a\in\Ocal_K[x]$ an irreducible polynomial, and $\gp \subset \Ocal_K$ a prime with residue characteristic $p$. We analyze the reducibility of $x^{p^m}-a=x^{p^m}-a_0\pi^{hp^k}$, where $v_\gp(a_0)=0$ and $\pi$ is a uniformizer at $\gp$. We write $w_0$ for $v_\gp\big(a_0^{p^{f-\mu+m}}-a_0\big)$ as in Definition \ref{Def: Wieferich}. Suppose the ramification index $e_\gp$ of $\gp$ over $p$ is not divisible by $p$. We define $l_0=\lceil\frac{w_0}{e_\gp}-\frac{p}{p-1}\rceil$. If $l_0 \leq 0$, then suppose $p\dnd w_0$, and if $l_0 < m$, then suppose $p \dnd \big(e_\gp l_0-w_0\big)$. 

Additionally, assume $K_\gp\cap \QQ_p\big(\zeta_{p^\infty}\big)=\QQ_p$, since the presence of $p^{\text{th}}$ roots of unity leads to excess splitting and is cumbersome to analyze in generality. 
Note that, by Theorem \ref{Thm: psplittingp^mgeneral}, $a_0\in K_\gp^{p^j}$ with $j>0$ if and only if $l_0\geq j$. 
We see that $a_0\in K_\gp^{p^{l_0}}$ and $a_0\notin K_\gp^{p^{l_0+1}}$. Further, $\pi^{hp^k}\in K_\gp^{p^j}$ if and only if $k\geq j$. Indeed, if $l_0\leq 0$ and $k>0$ or if $l_0>0$ and $k=0$, then $a_0\pi^{hp^k}\notin K_\gp^p$ and $x^{p^m}-a_0\pi^{hp^k}$ is irreducible by Theorem \ref{Thm: VahlenCapelli}. Further, if $l_0\leq 0$ and $k=0$, then taking the $x$-adic development shows $a_0\pi^{hp^k}\notin K_\gp^p$. (See Theorem \ref{Thm: pmidanotn}.) Again, $x^{p^m}-a_0\pi^{hp^k}$ is irreducible. Hence we focus on the case where $l_0>0$ and $k>0$. With these assumptions, let $s=\min(l_0,k)$. We see $a_0\pi^{hp^k}\in K_\gp^{p^{s}}$ but $a_0\pi^{hp^k}\notin K_\gp^{p^{s+1}}$.

Recall, the twisted cyclotomic polynomial are
\[\Phi_{p^j}\left(x,\sqrt[p^s]{a}\right)=\prod_{\substack{{1\leq k< p^j} \\ {\gcd(k,p)=1}}} x-\zeta_{p^j}^k \sqrt[p^{s}]{a}.\]
Using similar tactics to 
\cite{VelezPrimePower}, we will prove the following.

\begin{proposition}\label{Prop: Factorization}
With the notation as above and writing $c=\min(m,s)$,
\begin{equation}\label{Eq: Factorization}
x^{p^m}-a=\left(x^{p^{m-c}}- \sqrt[p^{c}]{a}\right) \prod_{i=1}^{c} \Phi_{p^i}\left(x^{p^{m-c}},\sqrt[p^{c}]{a}\right)
\end{equation}
is a factorization of $x^{p^m}-a$ into irreducibles in $K_\gp[x]$. If $c=0$, then we take the empty product to be 1. Moreover, each factor corresponds to a totally ramified extension of $K_\gp$.
\end{proposition}


To prove Proposition \ref{Prop: Factorization}, we will employ the following lemma. (This is essentially Lemma 4 of \cite{VelezPrimePower}.)

\begin{lemma}\label{Lem: TotalRam}
If $b\notin K_\gp^p$, then $K_\gp\big(\zeta_{p^i},\sqrt[p^j]{b}\big)$ is a totally ramified of degree $\phi(p^i)p^j$.
\end{lemma}


\begin{proof}[Proof of Lemma \ref{Lem: TotalRam}.]
We first claim that $\sqrt[p]{b}$ is not in any abelian extension of $K_\gp$. For a contradiction, suppose otherwise. This implies $K_\gp\big(\hspace{-.5ex} \sqrt[p]{b}\big)/K_\gp$ is abelian. Thus $\zeta_p\in K_\gp\big(\hspace{-.5ex} \sqrt[p]{b}\big)$, since $x^p-b$ splits. However, $\big[K_\gp\big(\zeta_p\big):K_\gp\big]\neq 1$ is a divisor of $p-1$ and $\big[K_\gp\big(\hspace{-.5ex} \sqrt[p]{b}\big):K_\gp\big]=p$. 
With this contradiction, we have our claim. 

Since $K_\gp\cap \QQ_p\big(\zeta_p^\infty\big)=\QQ_p$, we see that $K_\gp\big(\zeta_{p^i}\big)$ is totally ramified of degree $\phi(p^i)$. Since $K_\gp\big(\zeta_{p^i}\big)$ is abelian, $\sqrt[p]{b}\notin K_\gp\big(\zeta_{p^i}\big)$. Theorem \ref{Thm: VahlenCapelli} shows that $x^{p^j}-b$ is irreducible over $K_\gp\big(\zeta_{p^i}\big)$, so $\big[K_\gp\big(\zeta_{p^i},\sqrt[p^j]{b}\big):K_\gp\big]=\phi(p^i)p^j$. If $K_\gp\big(\zeta_{p^i},\sqrt[p^j]{b}\big)/K_\gp\big(\zeta_{p^i}\big)$ were not totally ramified, then we have an unramified extension $K_\gp\big(\zeta_{p^i}\big)\subset L \subset K_\gp\big(\zeta_{p^i},\sqrt[p^j]{b}\big)$ with $\big[L:K_\gp\big(\zeta_{p^i}\big)\big]=p$. Since unramified extensions of $p$-adic fields are given by adjoining roots of unity, we see $L/K_\gp$ is abelian. However, by degree considerations $x^{p^j}-b$ is reducible over $L$, so $\sqrt[p]{b}\in L$. This is a contradiction, and we find that $K_\gp\big(\zeta_{p^i},\sqrt[p^j]{b}\big)$ is totally ramified of degree $p^j$ over $K_\gp\big(\zeta_{p^i}\big)$. 
\end{proof}


\begin{proof}[Proof of Proposition \ref{Prop: Factorization}.] 
Indeed, the fact that Equation \eqref{Eq: Factorization} is a factorization is clear since $a\in K_\gp^{p^s}$. If $c=m$, then the result is clear from the hypothesis that $K_\gp~\cap~\QQ_p(\zeta_{p^\infty})~=~\QQ_p$. Hence, we assume $c=s$. 
Note that since $a=a_0\pi^{hp^k}\notin K_\gp^{p^{s+1}}$, we have $\sqrt[p^s]{a}~=~\sqrt[p^s]{a_0}\pi^{hp^{k-s}}~\notin~K_\gp^p$. Hence, Lemma \ref{Lem: TotalRam} yields the result.
\end{proof}

The correspondence between the factorization of $x^{p^m}-a$ in $K_\gp[x]$ the factorization of $\gp$ in $K\big( \sqrt[p^m]{a}\big)$ along with Proposition \ref{Prop: Factorization} yields the following result.


\begin{proposition}\label{Prop: DifficultCase}
With the setup established above, recall $a=a_0\pi^{hp^k}$, $w_0=v_\gp\big(a_0^{p^{f-\mu+m}}-a_0\big)$, $l_0=\lceil\frac{w_0}{e_\gp}-\frac{p}{p-1}\rceil$, $s=\min(l_0,k)$, and $c=\min(m,s)$. A prime $\gp$ of $\Ocal_K$ with odd residue characteristic factors\footnote{In contrast to Theorem \ref{Thm: psplittingp^mgeneral}, our labeling of ideals here agrees with the exponent of $p$ in the corresponding $p^i$-th twisted cyclotomic polynomial in \eqref{Eq: Factorization}.} as
\[\LPO{\gp}{K}{p^m}{a} = \gP^{p^{m-c}} \prod_{i=1}^{c}\gP_i^{p^{m-c}\varphi(p^i)},\]  
where $\varphi$ is Euler's phi function, and if $c\leq 0$, then we take the product to be 1. 
\end{proposition}


\begin{example}\label{Ex: 5^5*26}
Consider $f(x)=x^{25}-5^{5}\cdot 26$. We want to employ Proposition \ref{Prop: DifficultCase} to find how $5$ factors in $\QQ\big(\hspace{-.5ex} \sqrt[25]{81250}\big)$. We have $k=1$ and $w_0=v_5\big(26^5-26\big)=2$, so $l_0=1$ and $s=1$. As $m=2$, we have $c=1$. Thus, 
\[\LPO{5}{\QQ}{25}{81250} =\gP^{5^{2-1}} \prod_{i=1}^{1}\gP_i^{5\varphi(5^i)} =\gP^5\gP_1^{20}.\]
\end{example}


\begin{example}\label{Ex: p=3Knumberfield}
Let $K$ be any number field where the ramification index of each prime above 3 is relatively prime to 6. This ensures that $K_\gp\cap\QQ_3(\zeta_{3^\infty})=\QQ_3$ for any $\gp$ of $\Ocal_K$ above 3. Suppose $f(x)=x^{81}-82\cdot 3^9$ is irreducible in $K[x]$. Let $\gp\subset \Ocal_K$ be a prime above 3. For convenience, suppose $\Ocal_K/\gp\cong \FF_3$. We have $k=2$ and $w_0 = v_\gp(82^3-82)=3e_\gp$, so $l_0=2$. Thus $s=2$, and since $m=4$, we have $c=2$. Hence, the prime ideal factorization of $\gp$ in $K\big(\hspace{-.5ex} \sqrt[3^4]{82\cdot 3^9}\big)$ is
\[\gp \Ocal_{K\big(\hspace{-.5ex}\sqrt[3^4]{82\cdot 3^9}\big)}=\gP^{3^{4-2}} \prod_{i=1}^{2}\gP_i^{3^{4-2}\varphi(3^i)}=\gP^9\gP_1^{18}\gP_2^{54}.\]
\end{example}


\subsection{Extensions after the $p$-power extension}


Keeping the same notation, we recall that we wish to describe the splitting of $\gp$ in $K(\hspace{-.5ex} \sqrt[n]{a})$ where $n=n_0p^m$ with $\gcd(n_0,p)=1$. Given Proposition \ref{Prop: DifficultCase}, it suffices to describe the splitting of $\gP$, a prime of $K\big(\hspace{-.5ex} \sqrt[p^m]{a}\big)$ above $\gp$, in $K\big(\hspace{-.5ex} \sqrt[n]{a}\big)=K\big(\hspace{-.5ex} \sqrt[p^m]{a},\sqrt[n_0]{a}\big)$. Thus, we apply the Montes algorithm (Theorem \ref{Thm: Ore}) to $x^{n_0}-a$ over $K\big(\hspace{-.5ex} \sqrt[p^m]{a}\big)$.

Let $\pi_\gP$ be a uniformizer at $\gP$ and let $k_\gP$ be the residue field. 
Reducing $x^{n_0}-a$ at $\gP$, we have $x^{n_0}$. The $x$-adic development is simply $x^{n_0}-a$. The principal $x$-polygon is one-sided with slope $-\frac{v_\gP(a)}{n_0}$. Let $g_\gP=\gcd\big(n_0,v_\gP(a)\big)$. 
We consider the factorization of the residual polynomial associated to the lone side $S$:
\[R_S(y) = y^{g_\gP}-\frac{a}{\pi_{\gP}^{v_\gP(a)}} \ \text{ in } \ k_{\gP}[y].\]
Since $\gcd(n_0,p)=1$, this polynomial is separable. Thus, after accounting for the ramification index $\frac{n_0}{g_\gP}$, the ideal $\gP\Ocal_{K(\hspace{-.5ex}\sqrt[n]{a})}$ splits in $K\big(\hspace{-.5ex} \sqrt[n]{a}\big)$ in the same manner as $R_S(y)$ does in $k_\gP[y]$.

We have shown
\begin{proposition}\label{Prop: SplittingnotExplicit}
With the setup and notation as above, the ideal generated by $\gP$ factors as 
\[\LPO{\gP}{K}{n}{a}= \gI^{\frac{n_0}{g_\gP}},\]
where the prime ideal factorization of the ideal $\gI$ in $\Ocal_{K(\hspace{-.5ex} \sqrt[n]{a})}$ mirrors the factorization of the residual polynomial 
\[R_S(y) = y^{g_\gP}-\frac{a}{\pi_{\gP}^{v_\gP(a)}} \ \text{ in } \ k_\gP [y].\]
\end{proposition}


In conjunction with our work in previous sections, Proposition \ref{Prop: SplittingnotExplicit} completely describes the splitting of a prime $\gp$ of a number field $K$ in $K\big(\hspace{-.5ex} \sqrt[n]{a}\big)$ for a wide variety of $\gp$ and $K$. However, in the case where $K=\QQ$, we will improve the result by building explicit uniformizers to replace $\pi_\gP$ and further simplifying the residual polynomials. 

The first step is to work in a potentially smaller field. From Proposition \ref{Prop: Factorization}, we see that all the splitting at $p$ occurs in $\QQ\big(\hspace{-.5ex} \sqrt[p^c]{a}\big)$, and the extension $\QQ\big(\hspace{-.5ex} \sqrt[p^m]{a}\big)/\QQ\big(\hspace{-.5ex} \sqrt[p^c]{a}\big)$ is totally ramified of degree $p^{m-c}$ at the primes above $p$. Note $c=\min(w_0-1,k,m)$, since $l_0=w_0-1$ when $K=\QQ$. In $\QQ\big(\hspace{-.5ex} \sqrt[p^c]{a}\big)$, the splitting of $p$ mirrors the factorization of $x^{p^c}-a$ in $\QQ_p[x]$. Proposition \ref{Prop: Factorization} shows that we have the following factorization into irreducibles:
\begin{equation}\label{Eq: Factorizationc}
x^{p^c}-a=\left(x- \sqrt[p^{c}]{a}\right) \prod_{i=1}^{c} \Phi_{p^i}\left(x,\sqrt[p^{c}]{a}\right).
\end{equation}
As stated in Proposition \ref{Prop: DifficultCase}, the factorization of $x^{p^c}-a$ in $\QQ_p[x]$ in \eqref{Eq: Factorizationc} corresponds to a factorization of $p$ into primes in $\QQ\big(\hspace{-.5ex} \sqrt[p^c]{a}\big)$:
\[\LPO{p}{\QQ}{p^c}{a} =\gp_0 \prod_{i=1}^{c}\gp_i^{\varphi(p^i)}.\]
The completion of $\QQ\big(\hspace{-.5ex} \sqrt[p^c]{a}\big)$ at each $\gp$ is isomorphic to the extension of $\QQ_p$ obtained by adjoining a root of the corresponding irreducible factor of $x^{p^c}-a$. For example, $\QQ\big(\hspace{-.5ex} \sqrt[p^c]{a}\big)_{\gp_0}~\cong~\QQ_p\big(\hspace{-.5ex} \sqrt[p^c]{a}\big)$ and $\QQ\big(\hspace{-.5ex} \sqrt[p^c]{a}\big)_{\gp_1}\cong \QQ_p\big(\zeta_p\sqrt[p^c]{a}\big)$. However, Proposition \ref{Prop: Factorization} shows that $\sqrt[p^c]{a}\in \ZZ_p$. Thus, for our examples, $\QQ_p\big(\hspace{-.5ex} \sqrt[p^c]{a}\big)=\QQ_p$ and $\QQ_p\big(\zeta_p\sqrt[p^c]{a}\big)=\QQ_p(\zeta_p)$. In general, 
\begin{equation}\label{Eq: CycloCompletions}
\QQ\big(\hspace{-.5ex} \sqrt[p^c]{a}\big)_{\gp_i}\cong \QQ_p\big(\zeta_{p^i}\sqrt[p^c]{a}\big)=\QQ_p\big(\zeta_{p^i}\big).
\end{equation}
Therefore, a uniformizer for $\QQ\big(\hspace{-.5ex} \sqrt[p^c]{a}\big)_{\gp_0}$ is $p$, and fundamental results on cyclotomic fields show that a uniformizer for $\QQ\big(\hspace{-.5ex} \sqrt[p^c]{a}\big)_{\gp_i}$, with $i>0$, is $1-\zeta_{p^i}$.

Proposition \ref{Prop: SplittingnotExplicit} shows that the ramification index of any prime $\gP$ of $\QQ\big(\hspace{-.5ex} \sqrt[p^m]{a}\big)$ above $p$ in $\QQ\big(\hspace{-.5ex} \sqrt[n]{a}\big)$ is not divisible by $p$. Thus we can analyze the splitting of $p$ in $\QQ\big(\hspace{-.5ex} \sqrt[p^c]{a},\sqrt[n_0]{a}\big)$ and then multiply the ramification indices by a factor of $p^{m-c}$ to obtain the prime ideal decomposition of $p$ in $\QQ\big(\hspace{-.5ex} \sqrt[n]{a}\big)$. 

Summarizing the above discussion we obtain the following explicit description:

\vspace{.3 cm}

\noindent\textbf{Theorem \ref{Thm: Main} case \ref{MainIII}.} 
\textit{Suppose $x^n-a\in\ZZ[x]$ is irreducible. Let $p$ be an odd prime. Suppose $a=a_0p^{hp^k}$ and $n=n_0p^m$, where $\gcd(a_0,p)=\gcd(n_0,p)=1$ and $k,m>0$. Let $w_0=v_p\big(a_0^p-a_0\big)$, $c=\min(w_0-1,k,m)$, $g_0=\gcd(n_0,h)$, and $g=\gcd\big(n_0,h(p-1)\big)$. 
Then 
the ideal generated by $p$ experiences the following factorization
\[\LPO{p}{\QQ}{n}{a} = \gI_0^{\frac{p^{m-c} n_0}{g_0}}\prod_{i=1}^c\gI_i^{\frac{p^{m-c}\varphi(p^i) n_0 }{g}},\]
where the factorization of $\gI_0$ mirrors the factorization of 
\[R_{S_0}(y) = y^{g_0}-\frac{a}{p^{hp^k}} = y^{g_0}-a_0 \ \text{ in } \ \FF_p[y],\]
and the factorization of $\gI_i$ with $i>0$ mirrors the factorization of 
\[R_{S_i}(y) = y^{g}-\frac{a}{\left(1-\zeta_{p^i}\right)^{hp^k\varphi(p^i)}}=y^{g}-(-1)^{hp^k}a_0 \ \text{ in } \ \FF_p[y].\]}

\vspace{.3 cm}

Notice that $R_{S_i}(y)=R_{S_j}(y)$ for all $1<i,j\leq c$, so we denote this polynomial by $R_S(y)$.

\begin{proof}
The result is clear from Proposition \ref{Prop: SplittingnotExplicit} and the fact that the relevant completions in $\QQ\big(\hspace{-.5ex} \sqrt[p^c]{a}\big)$ are the $p$-power cyclotomic extensions of $\QQ_p$ described in \eqref{Eq: CycloCompletions}. However, the simplification 
$y^{g}-a/(1-\zeta_{p^i})^{hp^k\varphi(p^i)}=y^{g}-(-1)^{hp^k}a_0$ in $\FF_p[y]$ does take some argument. We have the unit \[\frac{p}{\left(1-\zeta_{p^i}\right)^{\varphi(p^i)}}=\prod_{\substack{{1\leq j< p^i} \\ {\gcd(j,p)=1}}}  \frac{1-\zeta_{p^i}^j}{1-\zeta_{p^i}}.\]
Long division with $x^j-1$ and $x-1$ yields 
\[\frac{1-\zeta_{p^i}^j}{1-\zeta_{p^i}}=\zeta_{p^i}^{j-1}+\zeta_{p^i}^{j-2}+\cdots+\zeta_{p^i}+1\equiv j\bmod \left( 1-\zeta_{p^i}\right).\]
Thus
\[\prod_{\substack{{1\leq j< p^i} \\ {\gcd(j,p)=1}}}  \frac{1-\zeta_{p^i}^j}{1-\zeta_{p^i}}\equiv \prod_{\substack{{1\leq j< p^i} \\ {\gcd(j,p)=1}}}  j \equiv \big((p-1)!\big)^{p^{i-1}}\equiv -1\bmod \left(1-\zeta_{p^i}\right). \eqno \qed\] 
\renewcommand{\qedsymbol}{}
\end{proof}


It is useful to have some examples in order to parse Theorem \ref{Thm: Main} case \ref{MainIII}. Though we could just apply the result directly, the following examples proceed in the spirit of the proof.

\begin{example}\label{Ex: ppower1}
We will look at $p=3$ and $n=3^4\cdot 2$, but we will vary the values of $a$. Take $a=3^3\cdot 5$. We see that $k=3$, $w_0=v_3(5^3-5)=1$, and $m=4$. Thus $c=\min(w_0-1, k,m)=0$. Beginning in $\QQ\big(\hspace{-.5ex} \sqrt[81]{135}\big)$, Proposition \ref{Prop: DifficultCase} shows $3\Ocal_{\QQ(\hspace{-.5ex} \sqrt[81]{135})}=\gp_0^{81}$. 
Moving to $\QQ\big(\hspace{-.5ex} \sqrt[162]{135}\big)$, we note $g_0=\gcd(n_0,h)=\gcd(2,1)=1$ so Theorem \ref{Thm: Main} case \ref{MainIII} shows $3 \Ocal_{\QQ(\hspace{-.5ex} \sqrt[162]{135})}=\gP^{162}$. 


We can take $a=3^6\cdot 5 = 3645$. We still have $k=1$, $w_0=1$, $m=4$, and $c=0$. Hence, with  $\QQ\big(\hspace{-.5ex} \sqrt[81]{3645}\big)$, we have $3\Ocal_{\QQ(\hspace{-.5ex} \sqrt[81]{3645})}=\gp_0^{81}$. However, $g_0=\gcd(n_0,h)=\gcd(2,2)=2$. Thus, Theorem \ref{Thm: Main} case \ref{MainIII} shows we must consider
\[R_{S_0}(y)=y^2-5 = y^2+1 \in \FF_3[y].\]
This polynomial is irreducible, so in the ring of integers of $\QQ\big(\hspace{-.5ex} \sqrt[162]{3645}\big)$ we have $3\Ocal_{\QQ(\hspace{-.5ex} \sqrt[162]{3645})}~=~\gP^{81}$ with $\gP$ having residue class degree 2. 

We change $a$ to $3^6\cdot 10=7290$. Still $k=1$ and $m=4$, but now $w_0=v_3(1000-10)=2$. Hence $c=1$, and Proposition \ref{Prop: DifficultCase} shows at $\QQ\big(\hspace{-.5ex} \sqrt[81]{7290}\big)$ we have $3 \Ocal_{\QQ(\hspace{-.5ex} \sqrt[81]{7290})}=\gp_0^{27}\gp_1^{54}$. Continuing, $g_0=\gcd(n_0,h)=\gcd(2,2)=2$ and $g=\gcd\big(n_0,h(p-1)\big)=\gcd(2,4)=2$, so
\[R_{S_0}(y)=y^2-10 = (y+1)(y-1) \in \FF_3[y] \ \text{ for } \ \gp_0,\]
and
\[R_S(y)=y^2-(-1)^6 10 = (y+1)(y-1) \in \FF_3[y]\ \text{ for } \ \gp_1.\]
Therefore, Theorem \ref{Thm: Main} case \ref{MainIII} shows 
\[\LPO{3}{\QQ}{162}{7290} =\gP_{0,0}^{27}\gP_{0,1}^{27}\gP_{1,0}^{54}\gP_{1,1}^{54}.\] 
Notice each factor has degree 1, so 3 is a common index divisor. To see this we could have simply used Corollary \ref{Cor: CIDs} to compute $d_{1,0}+\min(w_0-1,k,m)\cdot d_1=2+1\cdot 2 >\Irr(1,3)=3.$
\end{example}

To see the how the residual polynomials can vary, consider the following example.

\begin{example}\label{Ex: a=3 to the 2*27 times 80} 
Let $n=4\cdot 3^3$ and $a=3^{2\cdot 27}\cdot 80$, so we are considering the splitting of $3$ in $\QQ\big(\hspace{-.5ex} \sqrt[108]{3^{2\cdot 27}\cdot 80}\big)$. We have $m=3$, $k=3$, and $w_0=v_p\big(80^3-80\big)=4$. Hence, $c~=~\min(m,k,w_0-1)~=~3$ and Proposition \ref{Prop: DifficultCase} yields $3 \Ocal_{\QQ(\hspace{-.5ex} \sqrt[27]{a})}=\gp_0\gp_1^2\gp_2^6\gp_3^{18}$ in the extension $\QQ\big(\hspace{-.5ex} \sqrt[27]{a}\big)$. We have $g_0=\gcd(4,2)=2$ and $g=\gcd(4,2\cdot 2)=4$. Hence,
\[R_{S_0}(y)=y^2-80 = y^2+1\in \FF_3[y] \ \text{ for } \ \gp_0,\]
and
\[R_S(y)=y^{4}-(-1)^{2\cdot 27} 80 = y^{4}+1=(y^2 + y -1)(y^2 -y -1) \in \FF_3[y]\ \text{ for } \ \gp_1,\gp_2,\gp_3.\]

Thus, Theorem \ref{Thm: Main} case \ref{MainIII} shows
\[3 \Ocal_{\QQ\big(\hspace{-.5ex}\sqrt[810]{3^{5\cdot 27}\cdot 26}\big)}= \gP_0^{2}\gP_{1,0}^{2}\gP_{1,1}^{2}\gP_{2,0}^{6}\gP_{2,1}^{6}\gP_{3,0}^{18}\gP_{3,1}^{18} ,\]
where each $\gP$ has residue class degree 2. As there are only three monic, irreducible polynomials of degree 2 in $\FF_3[x]$, we see that 3 is a common index divisor due to the prime ideal factors of residue class degree 2. With Corollary \ref{Cor: CIDs}, the inequality is $d_{2,0}+\min(w_0-1,k,m)\cdot d_2=1+3\cdot 2 >\Irr(2,3)=3.$
\end{example} 

We undertake another, more involved example with $p=3$.

\begin{example}\label{Ex: a=3 to the 5*27 times 26} 
Take $n=3^4\cdot 10$ and $a=3^{5\cdot 27}\cdot 26$, so we are considering the splitting of $3$ in $\QQ\big(\hspace{-.5ex} \sqrt[810]{3^{5\cdot 27}\cdot 26}\big)$. As before, $m=4$, but now $k=3$ and $w_0=v_p\big(26^3-26\big)=3$. Hence, $c=2$ and Proposition \ref{Prop: DifficultCase} yields $3 \Ocal_{\QQ(\hspace{-.5ex} \sqrt[81]{a})}=\gp_0^{9}\gp_1^{18}\gp_2^{54}$. 
We have $g_0=\gcd(10,5)=5$ and $g=\gcd(10,5\cdot 2)=10$. Hence,
\[R_{S_0}(y)=y^5-26 = y^5+1\in \FF_3[y] \ \text{ for } \ \gp_0,\]
and
\[R_S(y)=y^{10}-(-1)^5 26 = y^{10}-1=(y^5+1)(y^5-1) \in \FF_3[y]\ \text{ for } \ \gp_1,\gp_2.\]
Factoring into irreducibles, 
\[y^5-1 = (y - 1)(y^4 + y^3 + y^2 + y + 1)  \text{ and }  y^5+1=(y + 1)(y^4 -y^3 + y^2 -y + 1)  \text{ in } \FF_3[y].\]
Thus, Theorem \ref{Thm: Main} case \ref{MainIII} shows that in the ring of integers of $\QQ\big(\hspace{-.5ex} \sqrt[810]{3^{5\cdot 27}\cdot 26}\big)$
\[3 \Ocal_{\QQ\big(\hspace{-.5ex}\sqrt[810]{3^{5\cdot 27}\cdot 26}\big)} = \gP_{0,0}^{18}\gP_{0,1}^{18}\gP_{1,0}^{18}\gP_{1,1}^{18}\gP_{1,2}^{18}\gP_{1,3}^{18}\gP_{2,0}^{54}\gP_{2,1}^{54}\gP_{2,2}^{54}\gP_{2,3}^{54},\]
where the residue class degrees $f_{*,*}$ are as follows: 
\[f_{0,0} ,f_{1,0} ,f_{1,1} ,f_{2,0} ,f_{2,1} =1 \ \text{ and } \ f_{0,1} ,f_{1,2} ,f_{1,3} ,f_{2,2} ,f_{2,3} =4.\]

There are five degree 1 factors of $3 \Ocal_{\QQ(\hspace{-.5ex}\sqrt[810]{3^{5\cdot 27}\cdot 26})}$, so 3 is a common index divisor. There are 18 degree 4 monic, irreducibles in $\FF_3[x]$, so the degree 4 factors do not provide an obstruction. Using Corollary \ref{Cor: CIDs}, we have $d_{1,0}+\min(w_0-1,k,m)\cdot d_1=1+2\cdot 2   >\Irr(1,3)=3,$ but $d_{4,0}+\min(w_0-1,k,m)\cdot d_4=1+2\cdot 2 \not>\Irr(4,3)=18.$

\end{example} 


\bibliography{Bibliography}
\bibliographystyle{alpha}


\end{document}